\documentclass[12pt]{article}
\usepackage{amsmath} 
\usepackage{amsthm}
\usepackage{amssymb}
\usepackage{color}

\usepackage{comment} 

\newcommand{\bK}{ 
{ \bf K}}
\newcommand{\bi}{\overline{\bf i }}
\newcommand{\w}{\omega}

\newcommand{\n}{^{(n)}}

\newcommand{\Z}{\ensuremath{\mathbb{Z}}}
\newcommand{\be}{\begin{equation}}
\newcommand{\ee}{\end{equation}}
\newcommand{\C}{\mathbb{C}}
\newcommand{\K}{\mathbb{C}}

\newcommand{\kxn}{k[x_1,\dots,x_n]}
\newcommand{\la}{\langle}
\newcommand{\N}{\mathbb{N}}
\newcommand{\R}{\mathbb{R}}
\newcommand{\ra}{\rangle}

\newcommand{\xb}{{\bf x}}

\newcommand{\yb}{{\bf y}}
\newcommand{\pb}{{\bf p}}

\renewcommand{\a}{\alpha}

\newcommand{\np}{\mathbb{N}_0}

\newcommand{\vp}{{\varphi}}

\newcommand{\mc}{\mathcal}
\newcommand{\mcm}{\mathcal M}
\newcommand{\pk}{{\bf p}^{(k)}}

\newcommand{\Q}{\mathbb{Q}}

\newcommand{\msl}{\mathsf{L}}

\newcommand{\vnj}{\mathcal{V}^n_j}
\newcommand{\gb}{Gr\"obner basis\ }
\newcommand{\zt}{I^{(\zeta)}_{\mathfrak{A}^{(n)}} }
\newcommand{\Supp}{\text{Supp}}

\newcommand{\dbtilde}[1]{\accentset{\approx}{#1}} 

\providecommand{\keywords}[1]
{
  \text{\textit{Keywords:}} #1
}

\providecommand{\msc}[1]
{
  \text{\textit{Mathematics Subject Classification (2020):}} #1
}

\newtheorem {thm}{Theorem}[section]
\newtheorem {pro}        [thm]{Proposition}
\newtheorem {cor}   [thm]{Corollary}
\newtheorem {lemma}       [thm]{Lemma}
\newtheorem {alg}        [thm]{Algorithm}
\newtheorem {conjecture}  [thm]{Conjecture}
\newtheorem {defin}         [thm]{Definition}

\theoremstyle{definition}
\newtheorem {remark}      [thm]{Remark}

\title{Invariants and reversibility\\ in polynomial systems of ODEs}
\author{Mateja Gra\v si\v c$^{1,2}$, Abdul Salam Jarrah$^3$ and Valery G.~Romanovski$^{1,4,5}$\\
$^1${\it Faculty of Natural Science and Mathematics,}\\ {\it  University of Maribor,
Koro\v ska cesta 160, SI-2000 Maribor, Slovenia}\\
$^2${\it Institute of Mathematics, Physics and Mechanics,}\\
{\it Jadranska 19, SI-1000 Ljubljana, Slovenia}\\
$^3${\it Department of Mathematics and Statistics}\\{\it American University of Sharjah,  Sharjah, UAE}\\
$^4${\it Faculty of Electrical Engineering and Computer Science,} \\ {\it University of Maribor,
 Koro\v ska cesta 46, SI-2000 Maribor, Slovenia}\\
$^5${\it Center for Applied Mathematics and Theoretical Physics,}\\
{\it Mladinska 3, SI-2000 Maribor, Slovenia}
}
\date{}

\begin{document} 
\maketitle

\begin{abstract}
This paper explores a relationship between invariants of certain group actions and the time-reversibility of two-dimensional polynomial differential systems exhibiting a $1:-1$ resonant singularity at the origin. We focus on the connection of time-reversibility with the Sibirsky subvariety of the center (integrability) variety, which encompasses systems possessing a local analytic first integral near the origin. An algorithm for generating the Sibirsky ideal for these systems is proposed  and  the algebraic properties of the ideal are examined. 

Furthermore, using a generalization of  the concept of time-reversibility we study  $n$-dimensional systems with a $1:\zeta:\zeta^2:\dots:\zeta^{n-1}$ resonant singularity at the origin, where $n$ is prime and $\zeta$ is a primitive $n$-th root of unity. We study the invariants of a Lie group action on the parameter space of the system, leveraging the theory of binomial ideals as a fundamental tool for the analysis. Our study reveals intriguing connections between generalized reversibility, invariants, and binomial ideals, shedding light on their complex interrelations.

\end{abstract}

\keywords{   Polynomial systems of ODEs, invariants, time-reversibility, symmetries, normal forms, binomial ideals}  

\msc{
34C14,  34C20, 13F65} 

\section{Introduction}

The theory of polynomial invariants provides a powerful framework for analyzing and understanding the behavior of solutions to differential equations. It offers insights into the global dynamics, stability, and structural properties of the solutions
and aids in classifying different types of differential equations based on their qualitative behavior. For instance, such invariants are 
crucial for the recent classification of quadratic systems presented in \cite{ALSV}.
Furthermore, polynomial invariants can  help to identify conservation laws
(which are first integrals from the mathematical point of view)
in physical and biochemical models 
described by differential equations. 

One of objectives of this work is to conduct a study of polynomial 
invariants following the direction proposed  in \cite{LiuLi,Sib1,Sib2}.  In these works, the authors focus on two-dimensional polynomial systems of ordinary differential equations (ODEs) that depend on parameters and can be represented in the form
\begin{equation} \label{ggs}
\begin{aligned}
\dot {x} &
        =  x - \sum_{(p, q) \in S} a_{pq} x^{p + 1} y^q, \\
\dot y &
        = - y + \sum_{(p, q) \in S} b_{qp} x^q y^{p + 1}\,,
\end{aligned}
\end{equation}
where $x$ and $y$ are complex variables (unknown functions).
The dot means the differentiation with respect to the variable $\tau$,
the coefficients $a_{pq}, b_{qp}$ are complex parameters, and 
\be \label{def:S}
S = \{(p_j,q_j) \mid p_j+q_j \geq 1, j = 1, \ldots, \ell \} \subset \N_{-1} \times \N_0,
\ee
where $\N_0=\N\cup \{0\}$ and $\N_{-1}=\N_0\cup \{ -1\}$.
The parameter space of \eqref{ggs} is  $\C^{2\ell}$,  and the ring of polynomials in  the variables $a_{p_1 q_1}, \dots, a_{p_\ell q_\ell}, 
b_{q_\ell p_\ell},\dots, b_{q_1 p_1}$ with coefficients in the field $\C$ will be denoted by $\C[a,b]$.

Let  $i=\sqrt{-1}$ and $\vp \in \R $ or $\vp \in \C$. The  action of the one-parameter group of transformations of the phase space  
\begin{equation} \label{ROT}
x' = e^{-i \vp} x, \quad  y' = e^{i \vp} y,
\end{equation}
defines the corresponding  action  on the space 
of parameters. Invariants of this  group action  were studied in \cite{LiuLi,Sib1,Sib2} with a special focus on their connection to the problem of 
 analytic integrability of vector field \eqref{ggs}. It was shown that the Lyapunov quantities 
 (also called saddle quantities, focus quantities \cite{RS}, and values of singular points \cite{LLH})
 are elements of the  polynomial subalgebra generated by  these invariants. It was proved  in \cite{LiuLi,RS}
that the coefficients of normal forms of  \eqref{ggs} have similar properties. 
These properties are valuable for the study of integrability, limit cycle bifurcations, and critical period bifurcations, as demonstrated in works such as \cite{FLRS, LLR, LPR}, along with the references provided therein.

Polynomial invariants are closely related to symmetries of differential equations. 
In this study, we specifically focus on what is known as the  time-reversible symmetry, a fundamental symmetry with significant relevance in various applications (see, for example,  \cite{AGG, BBT, Lamb, TM11, WRZ} and the references given  there). Additionally, we explore a generalization of this symmetry that was introduced in \cite{LPW, W}.

For  the $n$-dimensional system of ordinary differential equations
 \be \label{sys_X}
 \dot x = {\mc F}(x),
 \ee
 where ${\mc F}(x)$ is an $n$-dimensional vector of smooth functions
 defined on some domain $\Omega$ of $\R^n$ or $\C^n$, it 
is said  that the  system 
 is \emph{time-reversible} on $\Omega$ if there exists an
involution ${\Psi}$ defined on $\Omega$
 such that
\be \label{inv_gen}
D_\Psi^{-1} \cdot {\mc F} \circ \Psi = - {\mc F}.
\ee

For some families of systems \eqref{sys_X}, the time-reversibility yields the  analytic integrability  (see e.g. \cite{Bib,LPW,W}). In this paper we   will  consider an
 interconnection of invariants of group \eqref{ROT} and the time-reversibility in  system \eqref{ggs}, reviewing and further developing the results obtained in  \cite{JLR,R,RS,Sib1,Sib2}. 

For system \eqref{ggs},
our study begins by examining the invariants of a group associated with \eqref{ROT}, which results in an  algorithm for computing a basis of these invariants. Furthermore, we discuss some algebraic properties of the invariants  in view of the theory of binomial ideals. In particular, we show  that a binomial ideal generated 
by the invariants, the so-called Sibirsky ideal of system \eqref{ggs},  
is the lattice ideal and it defines the  set of time-reversible systems in the space of parameters of \eqref{ggs}.  Then we investigate similar problems for 
the $n$-dimensional polynomial systems
of the form 
\be \label{SYS}
\dot \xb =\mathcal{Z} {\bf x}+X({\bf x})=F({\bf x}),
\ee
where $n>2$,  ${\bf x}$ is either $(x_1,\dots, x_n)$ or its transpose $ (x_1,\dots, x_n)^\top$, 
$\mc Z$ is a diagonal matrix,  
\be \label{Zmatrix}
\mc Z= {\rm{diag}}[1, \zeta, \dots, \zeta^{n-1}]  ,
\ee
$n$ is a prime number, 
 $\zeta$ is an   $n$-th primitive root of unity and the components of the vector-function $X({\bf x})$ are polynomials  which do not contain constant and linear terms. 
The choice of this system is motivated by the results of \cite{LPW,W91,W}.
It was shown in \cite{W91}
that, for any linear system  
\be \label{As}
\dot \xb  = A_s \xb
\ee
where $A_s$ is a semisimple matrix, the algebra of polynomial first integrals of \eqref{As} is finitely generated. Moreover,  if  $ A_s$ 
has eigenvalues  $\lambda_1, \dots,  , \lambda_n, $
and the $\Z$-module spanned by all non-negative integer solutions $
(k_1, \dots , k_n) $ of
$ \sum_{i=1}^n
k_i \lambda_i = 0 $ has rank $d,$ then there are exactly $d$ independent polynomial first
integrals (and also exactly $d$ independent formal first integrals) for system \eqref{As}.

Under some conditions on the nonlinear system 
\be \label{Asn}
\dot \xb  = A_s \xb +X(\xb) = v(\xb),
\ee
where $X(\xb)$ are series without constant and linear terms, 
the  first integrals 
of \eqref{As}
are conserved in the Poincar\'e-Dulac   normal form of system \eqref{Asn} \cite{LPW,W91}.
One of such conditions is the generalized 
reversibility introduced in \cite{W,LPW},  which we will define in the next section. 

Clearly,  the algebra of polynomial first integrals 
of the linear system 
 \be \label{Zlin}
 \dot {\bf x}= {\mc Z} {\bf x}
 \ee 
is generated by the  single monomial
\be \label{I_int}
 \Phi({\bf x})= x_1 x_2 \cdots x_n.
\ee
Thus the study of local integrability of system
\eqref{SYS} represents the simplest 
generalization of the Poincar\'e  center problem
(see e.g. \cite{RS}) from the two-dimensional case 
to the $n$-dimensional case.

For system \eqref{SYS},
we first investigate  polynomial invariants of the
action of the  Lie  group of transformations 
\begin{eqnarray}\label{ch3d}
{\bf y}=e^{\mc Z\psi} {\bf  x}
\end{eqnarray}
on system \eqref{SYS},  where ${\bf y}=(y_1,\dots, y_n)$ and $\psi$
is a real or complex parameter, and give an algorithm to compute 
a basis of the polynomial subalgebra of these invariants. Then we show that the coefficients of the Poincar\'e-Dulac normal form are polynomials of the polynomial subalgebra of these invariants. 

Finally, we study properties 
of system \eqref{SYS} related to the reversibility
and their connection to the invariants. {In particular, we define 
a generalization of the Sibirsky ideal for a family of  systems  of the form \eqref{SYS}.  Unlike  the two-dimensional case,  it is not a lattice ideal in general. 
The main result of the paper is 
Theorem \ref{conj_2} which establishes a simple relation between the ideal defining the set of  invariants and the  ideal defining 
 the set of generalized reversible systems (which is  a subvariety of the integrability variety). To prove the theorem  we introduce a generalization of toric ideals. 
}


\section{Preliminaries}

Let $k$ be a field and $G$ a subgroup of the multiplicative group of invertible  $n \times n$ matrices with elements in $k$. 
For a matrix $A \in G$, 
let $A {\bf x}$ denote the usual action of $G$ on  $k^n$. 
A polynomial $f \in \kxn$ is  \emph{invariant under the action of the group $G$} (or simply \emph{an invariant of $G$}) if $f({\bf x}) = f(A {\bf x})$  for every ${\bf x} \in k^n$ and every $A \in G$. 


Applying transformation \eqref{ROT} to system  \eqref{ggs}  we  obtain the  system 
\[
\dot x' =x' - \sum_{(p,q) \in  S} a(\vp)_{pq} x'^{p+1}{y'}^q,
\quad  
\dot y' =  -y'+ \sum_{(p,q) \in  S} b(\vp)_{qp} x'^q{y'}^{p+1},
\]
with the coefficients of the transformed system being 
\be
\label{tr} 
a(\vp)_{pq}= a_{pq}  e^{i(p-q)\vp}, \quad
b(\vp)_{qp} = b_{qp} e^{-i(p-q)\vp}, 
\ee 
for all $(p,q)\in S$. 
Equations \eqref{tr} define a representation of group \eqref{ROT} in the space  $\C^{2\ell}$ of parameters of system \eqref{ggs}. 


The set of polynomials which are  invariant under this group action has been studied for the first time by Sibirsky \cite{Sib1,Sib2} and  latter by Liu and Li \cite{LiuLi}.
More precisely, Sibirsky considered the special case of system \eqref{ggs}, the complexifications of the real systems which can be written after re-scaling the time by $i$ as 
\begin{equation} \label{gsr}
\begin{aligned}
\dot x &
        = 
          x - \sum_{(p, q) \in S} a_{pq} x^{p + 1} \bar x^q,
\end{aligned}
\end{equation}
where $a_{pq}$ are complex parameters (see e.g. \cite[Chapter $3$]{RS} for more details).
Equation  \eqref{gsr} is a particular case of \eqref{ggs}   when in \eqref{ggs} $y=\bar x$ 
and the  equations on the right-hand side of \eqref{ggs} are  complex conjugate. 
We note that for  equation \eqref{gsr},  transformation
\eqref{ROT}
is just a rotation of the two-dimensional real phase plane $(u,v)$, where 
$x=u+iv$. Sibirsky called the studied invariants the {\it invariants of the rotation group}. 
However, if instead of \eqref{ROT}, we consider the group  of transformations 
\begin{equation} \label{ROT_r}
x' = e^{- \vp} x, \quad  y' = e^{ \vp} y,
\end{equation}
then it is easily seen that the invariants of \eqref{ROT_r} are the same as the ones of \eqref{ROT}. We will deal with group \eqref{ROT_r} instead of \eqref{ROT}.



Denote by $\nu$ the $2\ell$-tuple 
$(\nu_1, \ldots, \nu_{2\ell})$ in $\mathbb{N}_0^{2\ell}$.
Let 
$L: \np^{2\ell} \to \mathbb{Z}^2$ be a  homomorphism
of the additive monoids defined with respect to the ordered set 
$S$   given by \eqref{def:S} as  
\begin{equation} \label{Lnu}
L(\nu) = \binom{L^1(\nu)}{L^2(\nu)} 
       = \binom{p_1}{q_1} \nu_1        
       + \cdots
       + \binom{p_\ell}{q_\ell} \nu_\ell 
       + \binom{q_\ell}{p_\ell} \nu_{\ell+1}
       + \cdots 
       + \binom{q_1}{p_1} \nu_{2\ell}. 
\end{equation}

Let
\begin{equation} \label{lk}
\mathcal{M} = \bigcup_{K \in\mathbb{N}_0} \left \{ \nu \in\mathbb{N}_0^{2\ell} : L(\nu) = \binom{K}{K} \right\}.
\end{equation}
Obviously, $\mathcal{M}$ 
is an  Abelian monoid. Furthermore, for any $\nu \in \mathcal{M}$, we have $L^1(\nu)=L^2(\nu)$ and, hence, 
\begin{equation} \label{lsib} 
(p_1 -q_1) \nu_1 
+ \cdots 
+ (p_\ell-q_\ell)\nu_\ell 
+ (q_\ell-p_\ell) \nu_{\ell+1} 
+ \cdots
+(q_1-p_1) \nu_{2\ell} = 0. 
\end{equation} 
Since  $p_j+q_j\ge 1$ for all $1 \leq j \leq \ell$,  it is easy to see that 
the set of all non--negative integer solutions of equation \eqref{lsib}
coincides with the monoid ${\mathcal{M}}$ defined by equation \eqref{lk}.

For $ \nu \in \np^{2\ell}$ and the  ordered $2\ell$-tuple  of the parameters of system \eqref{ggs},  $$(a_{p_1 q_1},\dots,  a_{p_\ell q_\ell},
b_{q_{\ell} p_{\ell}},\dots, b_{q_1 p_1}),$$ we denote by $[\nu]$ the monomial 
\be \label{nu2d}
[\nu] = a_{p_1q_1}^{\nu_1} \cdots a_{p_\ell q_\ell}^{\nu_\ell}
         b_{q_\ell p_\ell}^{\nu_{\ell+1}} \cdots b_{q_1p_1}^{\nu_{2\ell}} .
\ee

It was observed by Sibirsky \cite{Sib1,Sib2} that a monomial $[\nu]$ is an invariant of group \eqref{ROT} if and only if $\nu \in \mathcal{M}$.
However, Sibirsky did not provide an algorithm for determining the  Hilbert basis -- a minimal spanning subset of the monoid $\mathcal{M}$  (it is known \cite{Sch,St} that such  basis does exist and is unique).
An efficient algorithm for identifying the  Hilbert basis of the monoid $\mathcal{M}$ has been proposed in \cite{JLR}. Another one follows from Algorithm 1.4.5 of \cite{St-AIT}. 



For  $\nu \in \N_0^{2\ell}$, 
 let $\hat \nu$ denote the involution of the vector $\nu$, 
\[
\hat \nu= (\nu_{2\ell},\nu_{2\ell-1},\ldots ,\nu_1).
\]

\begin{defin}
The ideal
\[
I_{S} = \langle [\nu]-[\hat \nu] : \nu\ \in \mathcal{M} \rangle \subset
\mathbb{C}[a,b]
\] 
 is called the \emph{Sibirsky ideal}
of system \eqref{ggs}.
\end{defin}

It was proved in \cite{JLR,R} that  the variety of the Sibirsky ideal,${\bf V}(I_{S})$, is the Zariski closure of the set of systems which are 
time-reversible with respect to the  linear
transformations  
\begin{equation} \label{xy} 
x \mapsto \alpha y, \ y\mapsto \alpha^{-1} x
\end{equation} 
for some $\a \in \C.$
Furthermore, it was proved in \cite{JLR} that all systems \eqref{ggs}, which  parameters belong to  ${\bf V}(I_{S})$, are locally analytically integrable in a neighborhood of the origin. 

Note that, for any $n$-dimensional  matrix of a cyclic permutation of order $n$, the eigenvalues of the matrix  are $e^{2i\pi k/n}$, where $k=0,1,\dots,n-1$.
Furthermore, transformation \eqref{xy} is a composition of the  orthogonal transformation with the matrix diag$[\alpha, \alpha^{-1}]$  
and the permutation with the matrix
\be \label{T2}
 T_2 = 
  \begin{pmatrix}
    0 & 1\\
   1 & 0
   \end{pmatrix},
\ee
and the diagonalization of  $T_2$ is the matrix 
 $$
 E_2 = 
  \begin{pmatrix}
    -1 & 0\\
   0 & 1
   \end{pmatrix}.
$$
Clearly, the Lie group  $e^{E_2\varphi}$  with the infinitesimal generator $E_2$
gives rise to the same group of transformations as \eqref{ROT_r}. 
In the next  section,  we  study some algebraic properties of the Sibirsky  ideal $I_S$ and  its connection to the invariants of the  group  $e^{E_2\varphi}$.

For  the $n\times n$ cyclic  permutation matrix 
\be \label{Tntilde}
\widetilde  T_n=\begin{bmatrix}
     0 & 1 & 0 & \ldots & 0 & 0\\
     0 & 0 & 1 & \ldots & 0 & 0\\
     0 & 0 & 0 & \ldots & 0 & 0\\
     \vdots & \vdots & \vdots & \ddots & \vdots & \vdots \\
     0 & 0 & 0 & \ldots & 0 & 1\\
     1 & 0 & 0 & \ldots & 0 & 0
 \end{bmatrix}, 
 \ee
the diagonalization  is the matrix $\mc Z$ given  by \eqref{Zmatrix}.
The  Lie group  $e^{{\mc Z}\psi} $  with the infinitesimal generator $\mc Z$  defines  transformation  \eqref{ch3d}
of the phase space of system \eqref{SYS} and induces a transformation (a representation) of the group in the space of parameters of the system.
We will study polynomial  invariants of the representation of  Lie group \eqref{ch3d} 
in the space of parameters of  a family 
of system \eqref{SYS} in Section \ref{sec_ndim}.

We denote by $T_n$ the transpose of \eqref{Tntilde}, that is, 
\be \label{Tn}
T_n= \widetilde{T_n}^\top.
\ee


It is easy to see that there are no
systems in family \eqref{SYS} that are time-reversible with respect to  linear involutions.  The notion, which we define below,  is a special case of the one introduced in \cite{WRZ} and is based on a  generalization of time-reversibility proposed in \cite{LPW,W}.  
\begin{defin}
 We say that system \eqref{SYS} is $(\widetilde  T_n,\zeta)$-reversible if 
\be \label{inv_gen_gen}
\zeta\,  
\widetilde T_n^{-1} \cdot F \circ \widetilde T_n(\xb) =   F(\xb), 
\ee
where  $\zeta^n=1$ and    $\zeta\ne 1$.
If \eqref{inv_gen_gen} holds for $\zeta=1$, we say that system \eqref{SYS} is $\widetilde T_n$-equivariant.
\end{defin}


 Let 
\be \label{tB}
{\mc B}=\text{diag}[\beta_1, \dots, \beta_n]
\ee
be an  $n\times n $ diagonal nonsingular matrix (that is,  $\beta_i\neq 0$ for all $i=1,2,\ldots,n$).

\begin{pro}
The 
transformation 
\be\label{xTy}
{\bf x}={\mc B} {\bf y},
\ee
where ${\mc B}$ is defined by \eqref{tB}, 
brings system \eqref{SYS}
to a system which is $(\widetilde T_n,\zeta)$-reversible,
if 
$$ B F({\bf x}) = \zeta
F( B {\xb}),
$$
where
\be \label{Bmatrix}
 B=\begin{bmatrix}
     0 & \frac{\beta_1}{\beta_2} & 0 & \ldots & 0 & 0\\
     0 & 0 & \frac{\beta_2}{\beta_3} & \ldots & 0 & 0\\
     0 & 0 & 0 & \ldots & 0 & 0\\
     \vdots & \vdots & \vdots & \ddots & \vdots & \vdots \\
     0 & 0 & 0 & \ldots & 0 & \frac{\beta_{n-1}}{\beta_n}\\
     \frac{\beta_n}{\beta_1} & 0 & 0 & \ldots & 0 & 0
 \end{bmatrix}.
\ee
\end{pro}
\begin{proof}
After applying  transformation \eqref{xTy} to  system \eqref{SYS} we obtain  the system 
\be \label{yF}
\dot {\bf y}={\mc B}^{-1} F({\mc B} \yb).
\ee
Performing the substitution 
$$
\yb=\widetilde T_n {\bf z}
$$
in system  \eqref{yF}  we obtain 
$$
\dot {\bf z}= ({\mc B}\widetilde T_n)^{-1} F(({\mc B} \widetilde T_n) {\bf z}).
$$

Thus, system \eqref{yF}
is $(\widetilde T_n,\zeta)$-reversible if 
$$
{\mc B}^{-1} F({\mc B} \yb)= \zeta({\mc B} \widetilde T_n)^{-1} F(({\mc B} \widetilde T_n) \yb),
$$
or, equivalently, 
$$
B F(\xb)=\zeta F(B \xb),
$$
where 
$B= {\mc B}\widetilde T_n  {\mc B}^{-1}.$ 
\end{proof}

From the above,  it is easy to see that  $\det B=1.$

\begin{defin} \label{defZrev}
For $\zeta\ne 1$ and $\zeta^n=1$, we say that system \eqref{SYS} is $\zeta$-reversible if 
\be \label{AF}
A F(\xb )=  \zeta F(A \xb)
\ee
for some matrix 
\be \label{Amatrix} 
 A=\begin{bmatrix}
     0 & \alpha_1 & 0 & \ldots & 0 & 0\\
     0 & 0 & \alpha_2 & \ldots & 0 & 0\\
     0 & 0 & 0 & \ldots & 0 & 0\\
     \vdots & \vdots & \vdots & \ddots & \vdots & \vdots \\
     0 & 0 & 0 & \ldots & 0 & \alpha_{n-1}\\
     \alpha_n & 0 & 0 & \ldots & 0 & 0
 \end{bmatrix}
 \ee
 where 
 \be \label{Al_p} \alpha_1 \alpha_2\cdots\alpha_n=1.  \ee
System \eqref{SYS} is called equivariant if
\eqref{AF}
holds for $\zeta=1.$
\end{defin}

\begin{pro} \label{prop25}
If system \eqref{SYS}
is $\zeta$-reversible, then it admits a local  first integral $\Psi(\xb)$ of the form 
\be \label{Psi_int}
\Psi(\xb)=x_1x_2\ldots x_n+h.o.t. 
\ee
in a neighborhood of the origin.
\end{pro}
\begin{proof}
For a  matrix  $A$ of the form \eqref{Amatrix} 
and a  nonsingular 
${\mc B}= {\rm diag}[\beta_1,\dots, \beta_n]$,
the equation 
\be \label{ABT}
A={\mc B}\widetilde T_n {\mc B}^{-1}
\ee 
yields 
$$\beta_{i+1}=\beta_i/\alpha_i \quad i=1,\dots, n-1, \quad \beta_n=\alpha_n/\beta_1.$$
Eliminating $\beta_i$ from the above equations, 
we see that \eqref{ABT} has a solution if condition 
\eqref{Al_p} is satisfied. 
Thus, if \eqref{AF} holds, then  system \eqref{yF} is $(\widetilde T_n,\zeta)$-reversible. Therefore, by Proposition 11 of \cite{LPW}, system \eqref{yF} admits a  first integral $\widehat  \Psi(\yb)=y_1y_2\ldots y_n+h.o.t.$ near the origin. Hence system \eqref{SYS} admits a  first integral of the form \eqref{Psi_int} in a neighborhood of the origin.
\end{proof}

\begin{remark}
First integral \eqref{Psi_int} can be analytic or merely formal.    
\end{remark}

In Section \ref{sec_ndim}
we investigate 
$ \zeta$-reversible systems and equivariant systems in 
 a  family  of polynomial systems  \eqref{SYS} 
 and study  an interconnection of $\zeta$-reversibility, equivariancy of  such systems and    invariants of Lie group \eqref{ch3d}.


\section{The two-dimensional  case }\label{sect2}

 {
 Two-dimensional system \eqref{ggs} looks like  a particular case of $n$-dimensional system
 \eqref{nSYS} studied in the next section. However in  the case $n=2$ the cyclic  permutation matrix
$\widetilde T_n$ defined by  \eqref{Tntilde} is an involution. 
For this reason the binomial ideals arising in the two-dimensional  case 
have some special features and admit a simpler consideration than the more 
general ideals arising in the next section. For example, Proposition \ref{pro35}, which is true in the case $n=2$, does not hold in the case $n>2$. 
For these reasons we treat the cases $n=2$ and $n>2$ separately. 
}

 
\subsection{Invariants of system \eqref{ggs}} 


Let $k$ be a field,    $n, d$ be positive integers,  where $n\geq d$, and $\mathfrak{A}$ be  an $d\times n$-matrix with integer elements and rank $d$. We can also write 
 $\mathfrak{A}=[\mathfrak{a}_1,\dots, \mathfrak{a}_n]$, where $\mathfrak{a}_i$ represents the $i$-th column of the matrix.
 
Let $t_1,\ldots,t_d,x_1,\ldots,x_n,z_1,\ldots,z_n$ be variables.  Fix any elimination monomial order with $\{t_1,\ldots,t_d\}\succ \{ x_1,\ldots,x_n\}\succ \{z_1,\ldots,z_n\}$ and  consider the $k$-algebra homomorphism
\begin{equation}\label{Hom_St}
\begin{array}{c} 
    {k}[x_1,\ldots,x_n,z_1,\ldots,z_n]\ \longrightarrow\ {k}[t_1,\ldots,t_d,t_1^{-1},\ldots,t_d^{-1},z_1,\ldots,z_n]\\
    x_i\ \mapsto\  z_i \displaystyle\prod_{j=1}^d t_j^{a_{ji}},\ \ z_i\ \mapsto z_i\ \ \text{ for all }\ i=1,\ldots, n.
\end{array}
\end{equation}

 The Hilbert basis $H_\mathfrak{A}$ of the  monoid ${\mathcal M}_\mathfrak{A}=\{\nu=(\nu_1,\dots, \nu_n)  \in {\mathbb N}_0^n\ :\  \mathfrak{A}\cdot \nu^\top=0\}$ can be obtained using the following algorithm, which is Algorithm 1.4.5 of
 \cite{St-AIT}.
\begin{alg}\label{alg1}
\begin{enumerate}
    \item[1)] Compute the reduced Gr\"obner basis $\mathcal G$ with respect to $\succ$ for the ideal 
$$
 \left \langle x_i -  z_i \displaystyle\prod_{j=1}^d t_j^{a_{ji}}\ : \ i=1,\dots, n\right\rangle. 
$$     
    \item[2)] The Hilbert basis $H_{
    \mathfrak{ A}}$ of ${\mathcal M}_\mathfrak{A}$ consists of all vectors $\nu$ such that $\mathbf{x}^{\nu}-\mathbf{z}^{\nu}$ appears in $\mathcal G$.
\end{enumerate}
\end{alg}

{

Let 
$k[{\bf t},{\bf t}^{-1}]:= k[t_1^\pm,\dots, t_d^\pm]$  be the Laurent polynomial algebra over $k$ in the variables $t_1 ,\dots  t_d $. Let us use the notation ${{\bf t}}^{\mathfrak{a}_i}=\displaystyle\prod_{j=1}^d t_j^{a_{ji}}$ for $i=1,\dots, n$, and define a $k$-algebra homomorphism
$$
\pi : k[\xb]\to k[{\bf t},{\bf t}^{-1}] \qquad \text{by} \qquad x_i\to {{\bf t}}^{\mathfrak{a}_i}. 
$$
The image of $\pi$, denoted by  $k[\mathfrak{A}]$, is the $k$-subalgebra  of  $k[{\bf t}, {\bf t}^{-1}]  $ and is called   the toric
ring of $\mathfrak{A}.$ The kernel of $\pi$ is denoted $I_\mathfrak{A}$ and called the {\it toric ideal} of $\mathfrak{A}$.

The Lawrence lifting of an integer matrix $\mathfrak{A}\in \Z^{d\times n} $ 
is the matrix 
$$
 \Lambda(\mathfrak{A})  = 
  \begin{pmatrix}
    \mathfrak{A} & 0\\
   E_n & E_n
   \end{pmatrix},
$$
where $E_n$ denotes the $n\times n$ identity matrix. It is clear that the 
\gb  $\mathcal{G}$ returned by Algorithm \ref{alg1} defines  the toric ideal $I_{\Lambda(\mathfrak{A})}$, which is the kernel of homomorphism \eqref{Hom_St}.

Denote by $\mathfrak{M}$ the $1\times 2\ell$ matrix on the left hand side of \eqref{lsib}, that is, 
\begin{equation} \label{lsib_M} 
\mathfrak{M}=[ (p_1 -q_1) \ 
\dots \ (p_\ell-q_\ell)\  
 (q_\ell-p_\ell) \ 
 \dots\ (q_1-p_1)].
\end{equation} 


Then  the kernal of the homomorphism 
given by
 \be \label{IrevT}
 a_{p_sq_s} \mapsto  y_s t^{q_s-p_s},\quad  b_{q_sp_s} \mapsto y_{2 \ell-s+1} t^{p_s-q_s}, 
 \qquad  s=1,\dots,\ell, \quad y_i\mapsto y_i,  \quad i=1,\dots, 2\ell, 
\ee
is the toric ideal $I_{\Lambda(\mathfrak{M})}$.


For  $\yb =(y_1, \dots, y_{2\ell})$,   we denote  by  $\yb^{\nu}$  the monomial $y_1^{\nu_1}y_2^{\nu_2}\ldots y_{2\ell}^{\nu_{2\ell}}$.
 We are interested in the 
Hilbert basis $H_\mathfrak{M}$ of the monoid $\mathcal M$ of solutions to  equation \eqref{lsib}.  
\begin{thm} \label{th_B}
Let 
$\mathcal{G}_{\mathfrak{M}}$ be the \gb of $ I_{\Lambda(\mathfrak{M})}$ 
computed according to 
Algorithm \ref{alg1}, in particular  $Y=\left\{ [\nu] - \yb ^{\nu}\ :\ \nu\in H_\mathfrak{M} \right\}$ is a subset of $\mathcal{G}_{\mathfrak{M}}$.
Then ${ G}=\{[\nu]-[\hat{\nu}]\ :
\ \nu\in H_\mathfrak{M}, \nu\neq \hat{\nu}\}$ is a Gr\"obner basis of $I_S$.
\end{thm} }
\begin{proof}
We first prove  that ${ G}$ is a  
basis of the ideal $I_S$.
Let $ [\mu]-[\hat{\mu}]\neq 0 $ ($ [\mu]-[\hat{\mu}]\not \in G  $) be an element of  the generating set $$\{ [\nu]- [\hat \nu] \ : \ \nu \in \mathcal{M}\} $$ of the ideal  ${ I}_S$. 
We show that $[\mu]-[\hat{\mu}]\in \la  G \ra.$
Since  $\mu \in \mcm$, $H_\mathfrak{M}$ is the  Hilbert basis of $\mcm$ 
and   $\mu\neq {\hat \mu}$, 
there exists a  $\nu\in H_\mathfrak{M}$ such that $[\nu]\mid [\mu]$.  Let $\sigma_1=\mu-\nu.$ Then $\sigma_1\in \mathcal{M}$ and  
\be \label{divmn}
[\mu]-[\hat \mu] =\frac{1}{2} ([\sigma_1]+[\hat \sigma_1]) ([\nu]-[\hat \nu])+ \frac{1}{2} ([\nu]+[\hat \nu]) 
([\sigma_1]-[\hat \sigma_1]), 
\ee
where  $\sigma_1<\mu$ with respect to the degree lexicographic monomial order (the deglex). 

If $[\sigma_1]-[\hat \sigma_1]\in G$, then 
$
 [\mu]-[\hat \mu]\in \la G \ra.  
$
If not, there is a $\nu_1\in H_\mathfrak{M} $  such that 
$[ \nu_1 ] | [\sigma_1] $. Letting $\sigma_2= \sigma_1-\nu_1$ and  continuing  the process we obtain  a sequence 
\be\label{seq_sig}
\mu > \sigma_1>\sigma_2>\sigma_3>\dots.
\ee 
If for  some $\sigma_k$  it holds $[\sigma_k]-[\hat \sigma_k]\in G$,  we are done. 
Otherwise the sequence \eqref{seq_sig} is infinite, which is impossible since the deglex is a term order
and, therefore, it is a well-ordering.
Thus,  $G$ is a basis for $I_S$. 

Since all monomials appearing in $G$ are relatively prime, $G$ is a 
Gr\"obner basis of $I_S$ (see e.g. Corollary 1.30 in \cite{HHO}).
\end{proof}

The set $ G$ contains both binomials $[\nu]-[\hat{\nu}]$ and $[\hat{\nu}]-[\nu]$. For every such pair of binomials, we delete from $G$ the binomial for which the leading coefficient is  $-1$  and denote the obtained set by ${ G}_{R}$. 
\begin{pro}\label{}
The set ${ G}_R$ is a reduced Gr\"obner basis of ${ I}_S$.
\end{pro}
\begin{proof}
Assume that for $[\mu]-[\hat{\mu}] \in {G}_R$ there exists $[\nu]-[\hat{\nu}]\in { G}_R$ such that 
\be \label{munu}
[\nu] \mid [\mu] \quad {\rm or}\quad  [\nu]\mid {[\hat{\mu}]}.
\ee 
By the definition of the set $G_R$, these $\nu,\mu,{\hat{\mu}}$ are in the Hilbert basis of $\mathcal M$, thus 
\eqref{munu}  contradicts  the minimality of the Hilbert basis.
\end{proof}

When we pass from $Y$ to $G$, the {\it self-conjugate} binomials, that is, binomials for which $\nu=\hat \nu$,  disappear.
Hence,  the Hilbert basis  $H_\mathfrak{M}$ is the union of the exponents that appear in $G_R$ and the exponents of the self-conjugate binomials, 
$$
{ H_\mathfrak{M}}=\{ \nu, \hat \nu   \, : \ [\nu]-[\hat\nu]\in G_R\}\cup \{\mathbf{e}_{j}+\mathbf{e}_{2\ell-j+1}\, :\ j=1,\ldots,\ell \ {\rm and}\ \pm ([\mathbf{e}_{j}]-[\mathbf{e}_{2\ell-j+1}])\notin { G_R}\}. 
$$

\begin{pro}\label{pro_m}
Let  $[\nu]-[\hat{\nu}]\in { G}_R$. Then there is no  $\theta \in \N_0^{2\ell}$ such that 
\be \label{eq18}
{[\theta] }\, \mid\left( [\nu]-[{\hat \nu}]\right).	
\ee
\end{pro}
\begin{proof}
We note that for any monomial $[\kappa] $  the exponent of  $GCD([\kappa],[\hat \kappa] )$ is a self-conjugate element of $\mcm$. This holds since   $GCD([\kappa], [\hat \kappa])=[\mu]$, where $\mu_i=\min\{\kappa_i,\kappa_{2\ell-i+1}\}$ for  $i=1,\ldots,2\ell$. Thus $\mu_i=\mu_{2\ell-i+1}$ for all $i=1,\ldots,\ell$ yielding  $\mu=\hat\mu\in {\mathcal M}$.  

Assume there exist $\nu\in H_{\mathfrak{M}}$ and $\kappa $ such that 
$$
[\nu]-[\hat{\nu}]=
{[\kappa] }  \left( [\mu]-[{\hat \mu}]\right)	
$$
for some  $[\mu]-[\hat{\mu}]$.  Since $[\kappa] | [\nu] $ and $[\kappa]|[\hat \nu]$,  $[\hat  \kappa] | [\nu] $ and $[\hat \kappa] |[\hat \nu]$, 
we conclude that $$ GCD([\kappa], [\hat \kappa]) | [\nu].$$ Since the exponent of  $GCD([\kappa], [\hat \kappa])$ is in $\mathcal{M}$ 
this contradicts to the minimality of $H_\mathfrak{M}$. 
\end{proof}

We now discuss some algebraic properties of the ideal $I_S$. 

For a field $k$, we denote by $k[\xb]:= k[x_1 ,\dots, x_n ]$ 
 the polynomial ring in  the variables $x_1 ,\dots , x_n$.
A subgroup $\mathsf{L}$ of the free abelian group $\Z^n$ is called a {\it lattice}. By definition,   a {\it basis} of a lattice $\mathsf{L}$ is a basis of the free abelian group $\mathsf{L}$. For $m \in \Z^{n}$ we  write 
$$
m=m^+-m^-,
$$
where $m^+, m^- \in \N^{n}_0$, and  denote by $f_m$ the binomial 
$\xb^{m^+}-\xb^{m^-}$. The ideal of $k[\xb]$ generated by all binomials $f_m$, where 
  $m\in \mathsf{L}$, 
 is the so-called {\it{lattice ideal}} of $\mathsf{L}$ \cite{HHO}.
 
In our case $n=2 \ell$ and the polynomial ring is $\C[a,b]$. Let $\mathsf{L}$ be the set consisting of   all  elements $\nu - \hat \nu $, where $\nu \in \mcm$. 
Since $\mathcal M$ is a monoid  closed under the action of involution on its elements, $\mathsf{L}$ is a subgroup of ${\mathbb Z}^{2\ell}$. 
Let us denote by $I_\msl$ the lattice ideal of  $\msl$. 

\begin{pro}\label{pro35}
$I_S=I_\msl$.
\end{pro}
\begin{proof}
We first show that $I_S\subset I_\msl$.  If $f=[\nu]-[\hat \nu]$ is a 	binomial from the \gb $G_R$ of $I_S$, then, 
by Proposition \ref{pro_m}, the monomials $[\nu]$ and $[\hat \nu]$ have a disjoint support. Therefore $f\in I_\msl$ yielding 
$I_S\subset I_\msl$. 

To prove the opposite  inclusion, let $\nu \in \mcm$. If 
$\nu $ and $\hat \nu $ have a disjoint support, then $f_{\nu-{\hat \nu}}=[\nu]-[\hat \nu] $, and hence  $f_{\nu-{\hat \nu}}\in I_S$. 
Otherwise, there is $\mu \in \mcm $ and $\kappa \in \mcm$ such that $\nu=\mu+\kappa$,  $\hat \nu= \mu +\hat \kappa$
and $\kappa $ and $\hat \kappa$ have a disjoint support.
But then  $f_{\kappa-{\hat \kappa}}=[\kappa]-[\hat \kappa] \in I_S $. Therefore $f_{\nu-{\hat \nu}}
=[\mu]  ([\kappa]  - [\hat \kappa]) \in I_S$. 
\end{proof}

\begin{cor}  
The Sibirsky ideal $I_S$ is a lattice ideal. 
\end{cor}   

{
\begin{remark}
    The fact that $I_S$ is a lattice ideal can be derived from the results of \cite{JLR} and \cite{HHO}, but we have presented a simple straightforward proof. 
\end{remark}
}

\subsection{Time-reversibility in system \eqref{ggs}} \label{sub32}

We now discuss a relation of the ideals introduced above to the time-reversibility. 
From  condition \eqref{inv_gen}, we see that system \eqref{ggs}
 is time-reversible with respect to transformation \eqref{xy} 
 if and only if 
 \be \label{abr}
b_{q_sp_s} =  \alpha^{p_s-q_s} a_{p_sq_s} ,  \quad  a_{p_sq_s} = b_{q_sp_s} \alpha^{q_s-p_s}\ \ 
 \text{ for } \   s=1,\dots, \ell.
\ee 

To find the set of all time-reversible systems in
the space of parameters  
$$
(a_{p_1q_1},\dots,a_{p_\ell q_\ell},b_{q_\ell p_\ell},\dots,b_{q_1p_1})\in \C^{2\ell} 
$$
of  system \eqref{ggs}   we need to find the variety of the elimination ideal
$$
\left\la  a_{p_sq_s} - b_{q_sp_s} \alpha^{q_s-p_s},\  b_{q_sp_s} - \alpha^{p_s-q_s} a_{p_sq_s}\ :\ s=1,\dots,  \ell\right\ra \cap \C[a,b]. 
$$ {
Clearly, the variety of this ideal is the same as the variety of 
\be \label{defI}
\mathcal{I}=J\cap \C[a,b], 
\ee
where
$$ 
J= \la a_{p_sq_s} - y_s t^{q_s-p_s},\quad  b_{q_sp_s} - y_{s}\ :\ s=1,\dots,\ell \ra. 
$$

We denote by $\mathcal{R}$ the set in   $\C^{2\ell}$  that corresponds to  
the set of all time-reversible systems of the form \eqref{ggs}.
Clearly, the Zariski closure  of the set $\mathcal{R}$ is the variety of the ideal $\mathcal{I}$.
Since $\mathcal{I}$ is a toric ideal, it is   a prime ideal.

The following theorem provides another way to compute 
the ideal $\mathcal{I}$ giving a characterization of $\mathcal{I}$ as a 
kernel of another rational map. 

\begin{thm}\label{th_2dim_rev}
 The ideal  $\mathcal{I}$ is the kernel of the polynomial map $\phi$ defined by 
    \be \label{Irev}
 a_{p_sq_s} \mapsto  y_s t^{q_s-p_s},\quad  b_{q_sp_s} \mapsto y_{2 \ell-s+1} t^{p_s-q_s}, 
 \quad y_{2\ell-s+1}\mapsto y_s \
 \text{ for } \ s=1,\dots,\ell.
    \ee
\end{thm}
\begin{proof}
Let 
\be \label{idU}
U= \left\la  a_{p_sq_s} -  y_s t^{q_s-p_s},\ b_{q_sp_s} -  y_{2 \ell-s+1} t^{p_s-q_s}, \
  y_{2\ell-s+1}- y_s\ : 
 \   s=1,\dots,\ell \right\ra.
\ee
By Proposition 3.5 of \cite{HHO},
$$\ker \phi= U\cap  \C[a,b]. $$
Setting in \eqref{abr} $\alpha =t^2$ and $ y_s= t^{q_s-p_s} b_{q_{s} p_{s}}, y_{2\ell -s+1} =t^{p_s-q_s} a_{p_sq_s} \text{ for } s=1,\dots, \ell$, we see that equations \eqref{abr} are equivalent to equations 
\be \label{aby}
 a_{p_sq_s} = y_s t^{q_s-p_s},    \quad  b_{q_sp_s} = y_{2 \ell-s+1} t^{p_s-q_s},   \quad
 y_s=y_{2\ell-s+1}\
 \text{ for } \ s=1,\dots,\ell.
\ee 

Since systems \eqref{abr} and 
\eqref{aby} are equivalent,  their varieties 
coincide. Furthermore, being toric ideals, both $\ker\, \phi$ and $\mc I$ are prime, and 
$$ 
U\cap \C[a,b]=J \cap \C[a,b]=\mathcal{I}.$$
\end{proof}
}
\begin{remark}\label{rem1}
We note that 
the kernel of the map \eqref{IrevT}
is the toric ideal of the Lawrence lifting 
$$
\Lambda(\mathfrak{M})= \begin{pmatrix} \mathfrak{M} & 0\\ E_{2\ell} & E_{2\ell} \end{pmatrix},
$$
where  $E_k$ is the notation for  the $k\times k $ identity matrix. 
The kernel of map \eqref{Irev} is the toric ideal of the matrix 
$$
\begin{pmatrix} \mathfrak{M} \\ E_{\ell} | \widehat E_{\ell} 
\end{pmatrix},
$$
where $ \widehat E_{\ell}$
is the ${\ell}\times{\ell}$
matrix having 1 on the secondary diagonal and the other elements equal to 0. 
\end{remark}

\begin{remark}
Setting $t=t_1^{-1} t_2$ in the ideal $U$ of \eqref{idU},  we obtain the ideal $\mathcal{J}$
of the Sibirsky subvariety algorithm of \cite{JLR}.
It was shown in \cite{JLR} that $\mathcal{J}\cap \C[a,b]=I_S.$ Thus, using, Theorem \ref{th_2dim_rev}, we have  
\be \label{i_i_s}
\mathcal{I}=I_S.
\ee 
Equality \eqref{i_i_s}
was proved earlier  in \cite{R}. Note that neither Proposition \ref{pro35} nor formula \eqref{i_i_s} hold in the case of system \eqref{nSYS} with  $n>2.$
\end{remark}

\section{The n-dimensional case} \label{sec_ndim}

Let $n $ be a prime number, $n>2$, and $\zeta$ be a primitive $n$-th root of unity. {Depending on the context, $\bar{\zeta}$ means either  the vector $(1,\zeta,\ldots,\zeta^{n-1})$ or $(1,\zeta,\ldots,\zeta^{n-1})^\top$. Let} 
$$
S\n =\{ {\bf{p}}^{(k)}=(p_1^{(k)},p_2^{(k)}\ldots,p_n^{(k)})^\top \in {\mathbb N}_{-1}\times {\mathbb N}_0^{n-1}\ :\ \sum_{j=1}^{n} p_j^{(k)}\geq 1, k=1,2,\ldots, \ell \}.
$$   
For an ordered $n$-tuple of indeterminants ${\bf x}=(x_1,x_2,\ldots,x_n)$ and ${\bf p}=(p_1,p_2,\ldots,p_n)^\top\in S\n $, let ${\bf x}^{{\bf{p}}}$ denote the monomial $x_1^{p_1}x_2^{p_2}\ldots x_n^{p_n}$.
We  consider a complex polynomial $n$-dimensional system of the form 
\begin{eqnarray}
\begin{aligned}\label{nSYS}
\dot x_1 &= x_1\left(1+\sum_{k=1}^{\ell}a^{(1)}_{ {\bf{p}}^{(k)}}\,{\bf{x}}^{{\bf{p}}^{(k)}} \right)\\
\dot x_2 &= x_2\left(\zeta+\sum_{k=1}^{\ell}a^{(2)}_{ T_n({\bf{p}}^{(k)})}\,{\bf{x}}^{T_n({\bf{p}}^{(k)})} \right)\\
\dot x_3 &= x_3\left(\zeta^2+\sum_{k=1}^{\ell}a^{(3)}_{ T_n^2({\bf{p}}^{(k)})}\,{\bf{x}}^{T_n^2({\bf{p}}^{(k)})} \right)\\
& \vdots \\
\dot x_n &= x_n\left(\zeta^{n-1}+\sum_{k=1}^{\ell}a^{(n)}_{ T_n^{n-1}({\bf{p}}^{(k)})}\,{\bf{x}}^{T_n^{n-1}({\bf{p}}^{(k)})} \right),
\end{aligned}\end{eqnarray}
 where the map $T_n$ denotes the cyclic permutation of components of $n$-tuple ${\bf{p}}\in S^{\n} $ determined by matrix \eqref{Tn}. In particular, 
\be \label{Tnmm} T_n({\bf{p}})=T_n(p_1,p_2,\ldots,p_n)^\top=(p_n, p_1, \ldots,p_{n-1})^\top
\ee
 and $T_n^m=T_n^{m-1}\circ T_n$, for $m=2,\ldots, n-1$. Note that $T_n^n$ is the identity map, and we  define $T_n^0=id$ by agreement. 

{
For instance, the 
  cubic system 
\be \label{sys_3dim_ex}
\begin{aligned}
\dot x= & x + a_{100} x^2 + a_{001} x z + a_{101} x^2 z,\\
\dot y= & \zeta y + b_{010} y^2 + b_{100} x y + b_{110} x y^2, \\
\dot z=& \zeta^2 z +c_{001} z^2 + c_{010} y z + c_{011} y z^2 
\end{aligned}
\ee
can be considered as a system  of the form \eqref{nSYS}. In this case 
$S=\{(1,0,0), \, (0,1,0),\,  (0,0,1)\}. $
The subscripts of the parameters of the first equation are defined by $S.$
The subscripts of the parameters of the second and the third equations of \eqref{sys_3dim_ex}
are obtained by applying $T_3$ and $T^2_3$, respectively, to the elements of $S$. To write \eqref{sys_3dim_ex} precisely in the form \eqref{nSYS}
one sets $   a_{100} =  a_{100}^{(1)}, 
$ $b_{010}=  a_{010}^{(2)},$ $ c_{001}
=a_{001}^{(3)}$  and so on.
}

 Clearly, system \eqref{nSYS} follows the form \eqref{SYS} and can be regarded as a generic family of $n$-dimensional polynomial systems with the linear part \eqref{Zlin}. This is because any polynomial system with the linearization defined by \eqref{Zlin} can be included into  a family \eqref{nSYS}.
{ Observe also that system \eqref{ggs} can be considered as a particular case of system \eqref{nSYS}, if we allow  $n=2$ in \eqref{nSYS}. }

We can write system \eqref{nSYS} in a short form 
as 
$$
\dot {\bf x}= {\mc Z} {\bf x} +\xb \odot R(\xb),
$$
where 
$$
R(\xb)= \left( \sum_{k=1}^{\ell}a^{(1)}_{ {\bf{p}}^{(k)}}{\bf{x}}^{{\bf{p}}^{(k)}}, \sum_{k=1}^{\ell}a^{(2)}_{ T_n({\bf{p}}^{(k)})}{\bf{x}}^{T_n({\bf{p}}^{(k)})}\, , \ldots \ , \sum_{k=1}^{\ell}a^{(n)}_{ T_n^{n-1}({\bf{p}}^{(k)})}{\bf{x}}^{T_n^{n-1}({\bf{p}}^{(k)})} \right)^\top
$$
and $\odot$ stands for the Hadamard product. 



\subsection{Invariants of system \eqref{nSYS}}

In this subsection  we study polynomial invariants of the action of transformation 
\eqref{ch3d} on system \eqref{nSYS} generalizing results of \cite{LiuLi,RS,Sib1,Sib2} to the $n$-dimensional case. 

Consider the transformation of the phase space
defined by \eqref{ch3d}, that is, $\yb=e^{\mathcal{Z} \psi} \xb $ 
for   $\psi\in {\mathbb C}$.
Notice that, for every ${\bf p}\in S\n$,
\begin{equation}
(e^{{\psi}{\mc Z}}{\bf x})^{\bf p}=
e^{\psi \,  {\bf p} \cdot\bar \zeta } {\bf x}^{\bf p},
\end{equation}
where $ {\bf p}\cdot \bar{\zeta}$ denotes the inner product of vectors   ${\bf p}$ and $\bar{\zeta}$. 
Analogous formula holds if $\bf p$ is substituted with any of $T_n^m({\bf p})$, $m=1,2\ldots,n-1$. Using these relations, it can be seen that transformation  \eqref{ch3d} brings system \eqref{nSYS} into the form
 \be \label{nSYS_transf}
\dot \xb= \mathcal{Z} \xb+ \xb \odot R_{\psi}(\xb),
\ee
where $R_{\psi}(\xb)$ denotes the vector
\begin{multline}\label{n_dimChange}
 \left( \sum_{k=1}^{\ell}a^{(1)}_{{\bf{p}}^{(k)}} {e}^{-\psi\,{\bf{p}}^{(k)}\cdot {\bar \zeta}} {\bf x}^{{\bf{p}}^{(k)}}, \sum_{k=1}^{\ell}a^{(2)}_{ T_n({\bf{p}}^{(k)})} {e}^{-\psi\,T_n({\bf{p}}^{(k)})\cdot {\bar \zeta}} {\bf x}^{T_n({\bf{p}}^{(k)})},\,\ldots \right.\\
\left.
\ldots,\sum_{k=1}^{\ell}a^{(n)}_{ T_n^{n-1}({\bf{p}}^{(k)})} {e}^{-\psi\,T_n^{n-1}({\bf{p}}^{(k)})\cdot {\bar \zeta}} {\bf x}^{T_n^{n-1}({\bf{p}}^{(k)})} \right)^\top
\end{multline}
representing the change of parameters of  system \eqref{nSYS}.

Define the ordered $n \ell$-tuple  of parameters of system \eqref{nSYS}
as
\begin{multline} \label{Apar}\left( a^{(1)}_{{\bf{p}}^{(1)}},a^{(2)}_{T_n({\bf{p}}^{(1)})},\ldots, a^{(n)}_{T_n^{n-1}({\bf{p}}^{(1)})}, a^{(1)}_{{\bf{p}}^{(2)}},a^{(2)}_{T_n({\bf{p}}^{(2)})},\ldots, a^{(n)}_{T_n^{n-1}({\bf{p}}^{(2)})},\ldots \right.\\
\left.
\ldots, a^{(1)}_{{\bf{p}}^{(\ell)}},a^{(2)}_{T_n({\bf{p}}^{(\ell)})},\ldots, a^{(n)}_{T_n^{n-1}({\bf{p}}^{(\ell)})}\right)\end{multline}
and consider the algebra of complex polynomials $\K[a]$, where the variables 
are the parameters $a^{(1)}_{{\bf{p}}^{(1)}},\ldots,a^{(n)}_{T_n^{n-1}({\bf{p}}^{(\ell)})}$ of system \eqref{nSYS}.
For the ordered $n \ell$-tuple \eqref{Apar}  
and $\nu=\left( \nu_1, \nu_2,\ldots,\nu_{n\ell}\right)\in {\mathbb N}_0^{n\ell}$, the  monomial $[\nu]$ in ${\mathbb C}[a]$ is  defined by 
\begin{eqnarray*}
\begin{aligned}
\relax [\nu]=& \prod_{m=0}^{n-1} {a^{(m+1)}_{ T_n^m({\bf{p}}^{(1)})}}^{\nu_{m+1}}\prod_{m=0}^{n-1} {a^{(m+1)}_{ T_n^m({\bf{p}}^{(2)})}}^{\nu_{n+m+1}} \cdots \prod_{m=0}^{n-1} {a^{(m+1)}_{ T_n^m({\bf{p}}^{(\ell)})}}^{\nu_{(\ell-1)n+m+1}}\\
    =&\prod_{k=1}^{\ell}\left( \prod_{m=0}^{n-1} {a^{(m+1)}_{ T_n^m({\bf{p}}^{(k)})}}^{\nu_{(k-1)n+m+1}}\right).
\end{aligned}
\end{eqnarray*}


The change of parameters \eqref{n_dimChange} of system \eqref{nSYS} induces (in a natural way) a $\mathbb C$-algebra endomorphism $\mathcal T$ on the algebra $\C[a]$, defined on a monomial $[\nu]$ as
\begin{equation} \label{invar}
{\mathcal T}[\nu]=[\nu]e^{-\psi\, {\bf \Sigma}(\nu)},
\end{equation}
where
\begin{eqnarray*} \label{sigma}
    \begin{aligned} 
        {\bf \Sigma}(\nu)=\sum_{k=1}^{\ell}\left(\sum_{m=0}^{n-1} \left(T_n^m({\bf{p}}^{(k)})\cdot \bar{\zeta}\right) \nu_{(k-1)n+m+1}\right).
    \end{aligned}
\end{eqnarray*}

We can rewrite the above  expression as 
\begin{eqnarray} \label{sigmaL} 
    \begin{aligned}
        {\bf \Sigma}(\nu)=L^1(\nu) +L^2(\nu)\zeta + L^3(\nu)\zeta^2+ \cdots + L^n(\nu)\zeta^{n-1},
    \end{aligned}
\end{eqnarray}
where $L^j(\nu)\in {\mathbb Z}$ for every $j=1,2,\ldots,n$. More explicitly,
\begin{equation} \label{L_ndim} 
    \begin{aligned} 
        L^1(\nu)&=\sum_{k=1}^{\ell}\left( p_1^{(k)}\nu_{(k-1)n+1}+p_n^{(k)}\nu_{(k-1)n+2}+p_{n-1}^{(k)}\nu_{(k-1)n+3}+\cdots +p_2^{(k)}\nu_{kn} \right)\\
        L^2(\nu)&= \sum_{k=1}^{\ell}\left( p_2^{(k)}\nu_{(k-1)n+1}+p_1^{(k)}\nu_{(k-1)n+2}+p_n^{(k)}\nu_{(k-1)n+3}+\cdots +p_3^{(k)}\nu_{kn} \right)\\
        &\vdots\\
        L^n(\nu)&=\sum_{k=1}^{\ell}\left( p_n^{(k)}\nu_{(k-1)n+1}+p_{n-1}^{(k)}\nu_{(k-1)n+2}+p_{n-2}^{(k)}\nu_{(k-1)n+3}+\cdots +p_1^{(k)}\nu_{kn} \right).
    \end{aligned}
\end{equation}
Map  \eqref{L_ndim} can be written in a more compact form as follows. 
Define the $n\times n \ell$ matrix $\mathfrak{L}$ as
$$
\left[ \pb^{(1)} \,\,  T_n(\pb^{(1)}) \,\,  \dots \,\, 
T_n^{n-1}(\pb^{(1)}) \,\,  \pb^{(2)} \,\,   T_n(\pb^{(2)}) \,\,  \dots \,\, 
T_n^{n-1}(\pb^{(2)}) \,\,  \dots \,\,  \pb^{(\ell)} \,\,    T_n(\pb^{(\ell)})  \,\, \dots  \,\, 
T_n^{n-1}(\pb^{(\ell)})\right].
$$ 
Then 
\begin{equation} \label{defL}
L(\nu) := \left(L^1(\nu),L^2(\nu),\ldots,L^n(\nu)\right)^\top
=\mathfrak{L} \nu^\top.
\end{equation}


Considering the vectors 
$$
\widetilde{\bf{p}}^{(k)}=(p_n^{(k)},p_{n-1}^{(k)},\ldots ,p_1^{(k)})^\top \, \mbox{ and } \, 
 \nu^{\{k\}}=(\nu_{(k-1)n+1},\nu_{(k-1)n+2},\ldots ,\nu_{kn} )^\top,
$$
where  $1\leq k \leq \ell$, and using the inner product notation, 
$L^j(\nu)$ can be written as
\begin{eqnarray} \label{L^j}
    \begin{aligned}
        L^j(\nu)&=\sum_{k=1}^{\ell} T_n^j(\widetilde{\bf{p}}^{(k)})\cdot \nu^{\{k\}} \ \ \text{ for every }\ j=1,2,\ldots, n.
    \end{aligned}
\end{eqnarray}

\begin{remark}
 We note that there is a discrepancy in the definition 
 of the map $L(\nu)$ in this section and in the previous sections.  
 To be consistent with the definition given by \eqref{defL},
map \eqref{Lnu} has to be written as 
\begin{multline*} 
L(\nu) = \binom{L^1(\nu)}{L^2(\nu)}
       = \binom{p_1}{q_1} \nu_1        
       + \binom{q_1}{p_1} \nu_{2}
       + \binom{p_2}{q_2} \nu_3        
       + \binom{q_2}{p_2} \nu_{4}
       +
       \cdots \\
       +\binom{p_{2 \ell-1}}{q_{2 \ell-1}} \nu_{2 \ell-4} 
       +\binom{q_{2\ell-1}}{p_{2\ell-1}}  \nu_{2\ell-3}
    +\binom{p_{2\ell}}{q_{2\ell}} \nu_{2\ell-1}       
    +\binom{q_{2\ell}}{p_{2\ell}} \nu_{2\ell}. 
\end{multline*}
 However we have used the definition given by \eqref{Lnu}
 to be consistent with the notations in the works \cite{JLR,R,RS}
\end{remark}

With  the introduced   notation, equation \eqref{sigma} can be written as 
\be \label{Lmap}
{\bf \Sigma}(\nu)=
        \sum_{k=1}^{\ell} \sum_{j=1}^{n} \left(T_n^j(\widetilde{\bf{p}}^{(k)})\cdot \nu^{\{k\}}\right)\,\zeta^{j-1}.
\ee 


Notice that, using  the notation of map $L$ given by \eqref{defL} and \eqref{L^j}, equation \eqref{Lmap} can be written as inner product, that is, 
\be \label{Lmap1}
{\bf \Sigma}(\nu)= L(\nu)\cdot \bar{\zeta}. 
\ee 

Using  map \eqref{defL} we  define the monoid
\begin{equation} 
\mathcal M\n=\bigcup_{K\in {\mathbb N}_0} \left \{ \nu\in {\mathbb N}_0^{n\ell}\ :\ L(\nu)=(K,K,\ldots,K)^\top \right \}.
\end{equation}
It is easy to see that $\mathcal M\n$ is a submonoid of ${\mathbb N}_0^{n\ell}$. 
\begin{thm} \label{THMinvarM}
A monomial $[\nu] $ is an invariant of  group \eqref{ch3d} if and only if 
$\nu \in \mcm\n$. 
\end{thm}
\begin{proof}
First we  note that a monomial $[\nu]$ is an invariant of group \eqref{ch3d} if and only if ${\mathcal T}[\nu]=[\nu]$. 

If $\nu\in \mcm\n$, then by \eqref{sigmaL},
${\bf \Sigma} (\nu)=K(1+\zeta+\zeta^2+\cdots+\zeta^{n-1})$, for some $K\in {\mathbb N}_0$. 
Since $1+\zeta+\cdots+\zeta^{n-1}=0$, 
we further see that ${\bf \Sigma} (\nu)=0$ and, therefore, ${\mathcal T}[\nu]=[\nu]$.

For the other direction, let $\nu$ be such that ${\mathcal T}[\nu]=[\nu]$. By \eqref{invar}, it follows that ${\bf \Sigma} (\nu)=0$.  
Using \eqref{sigmaL}, we have
\begin{eqnarray} \label{Sigma=0}
    \begin{aligned}
        L^1(\nu) +L^2(\nu)\zeta + L^3(\nu)\zeta^2+ \cdots + L^n(\nu)\zeta^{n-1}=0.
    \end{aligned}
\end{eqnarray}

Let us consider two cases. If $L^1(\nu)=0$, then $L^2(\nu)=L^3(\nu)=\ldots=L^n(\nu)=0$, since $\zeta$ is a primitive $n$-th root of unity with $n$ being a prime number 
and therefore the set $\{ \zeta, \zeta^2, \dots, \zeta^{n-1}   \}$
is linearly independent over $\Q.$
Therefore $\nu\in {\mathcal M}\n$. 
Similarly, if  $L^n(\nu)=0$, then  $L^i(\nu)=0 $ for $i=1, \dots, n-1 $ and  $\nu\in {\mathcal M}\n$.

Assume now that $L^1(\nu)\neq 0$. 
Let us denote $L_j=\frac{L^j(\nu)}{L^1(\nu)}$ for every $j=2,\ldots, n$. 
By \eqref{Sigma=0}, we have
 \be \label{eL1}
 1+L_2\zeta +L_3\zeta^2 + \cdots + L_n\zeta^{n-1}=0.
 \ee 
 Notice that $L_n\neq 0$. Multiplying 
 the latter equality by $\zeta$ and dividing by $L_n$,
we obtain
 \be \label{eL2}
 1+\frac{1}{L_n}\zeta +\frac{L_2}{L_n}\zeta^2 + \cdots + \frac{L_{n-1}}{L_n}\zeta^{n-1}=0.
 \ee
Subtracting \eqref{eL2} from \eqref{eL1} and using the above mentioned fact about $\zeta$, we see that 
$$L_2=\frac{1}{L_n}, L_3=\frac{L_2}{L_n}, \ldots, L_n=\frac{L_{n-1}}{L_n},$$
which imply $(L_n)^n=1$. Since $L_n$ is a rational number and $n>2$ is prime, it follows that $L_n=1$ and, therefore, $L_j=1$ for every $j=2,\ldots,n-1$. Thus, $L^1(\nu)=L^2(\nu)=\cdots =L^n(\nu)=K$. It remains to prove that $K\in \mathbb N$. But this holds true since, for  $\nu\in {\mathbb N}_0^{n\ell}$, $$nK=L^1(\nu)+L^2(\nu)+\cdots +L^n(\nu)=\sum_{k=1}^{\ell}\left(\left(\sum_{j=1}^{n} p_j^{(k)}\right)\left(\sum_{j=1}^n \nu_{(k-1)n+j}\right)\right)$$ and $\displaystyle \sum_{j=1}^{n} p_j^{(k)}\geq 1$ for all $k=1,2,\ldots, \ell$. 
\end{proof}

Consider the cyclic permutation matrices $\widetilde T_n$ and $T_n$ 
as defined by \eqref{Tntilde} and  \eqref{Tn}, respectively,   and let $\widetilde{\bf T}, {\bf T}$ be the $n\ell \times n\ell$ block diagonal matrices 
$$
\widetilde{\bf T}= {\rm{diag}}[ \widetilde T_n, \widetilde T_n,\ldots,\widetilde T_n] \ \ \text{ and } \ \ {\bf T}={\rm{diag}}[T_n,T_n,\ldots, T_n].
$$
Notice that $\widetilde{\bf T}$ is the inverse matrix of ${\bf T}$ and that ${\bf T}=\widetilde{\bf T}^{\top}$, thus they are unitary matrices. 
For $\nu \in {\mathbb N}_0^{n\ell}$ and integer $k$, define
\be \label{nuk}
\tilde   {\nu}^{(k)}:=\nu {\bf T}^k .
\ee


\begin{defin}
Let $\nu \in {\mathbb N}_0^{n\ell}$.\begin{enumerate}
\item We say that a
monomial $[\tilde{\nu}^{(k)}]$ is conjugate to the monomial $[\nu]$ if 
$\tilde{\nu}^{(k)} $ satisfies 
\eqref{nuk} for some $k=1,\dots, n-1.$ If $\nu = \nu {\bf T}$, we say that the monomial $[\nu]$ is self-conjugate. 
\item By $\tilde \nu$ we denote any element of the set  $\{  \tilde \nu^{(1)}, \dots , \tilde \nu^{(n-1)} \}$.
\end{enumerate}\end{defin}

Note that for the additive maps $L^j$, $j=1,2,\ldots,n$, defined by \eqref{L_ndim}, the following equations hold 
\begin{equation}\label{L_cyc}
L^j(\tilde{\nu}^{(k)})=L^{](j-1)+ k[_n+1}(\nu) \ \text{ for every }\ k=1,2,\ldots,n-1,
\end{equation}
where for an integer $m$ by $]m[_n$ we denote the  remainder of  division $m$ by $n$.
Therefore 
\begin{equation*} 
    \begin{aligned}
       &\left(L^1(\nu),L^2(\nu),\ldots,L^n(\nu)\right)
       = \left(L^n(\tilde{\nu}^{(1)}),L^1(\tilde{\nu}^{(1)}),\ldots ,L^{n-1}(\tilde{\nu}^{(1)})\right)\\
       &=\left(L^{n-1}(\tilde{\nu}^{(2)}),L^n(\tilde{\nu}^{(2)}),\ldots ,L^{n-2}(\tilde{\nu}^{(2)})\right)=\ldots
       =\left(L^{2}(\tilde{\nu}^{(n-1)}),L^3(\tilde{\nu}^{(n-1)}),\ldots ,L^{1}(\tilde{\nu}^{(n-1)})\right).
    \end{aligned}
\end{equation*}
Using the 
 cyclic permutation map $T_n$ defined by \eqref{Tnmm}, this can be written as
\begin{equation*}
L(\nu)=T_n(L(\tilde{\nu}^{(1)}))=T_n^2(L(\tilde{\nu}^{(2)}))=\ldots=T_n^{n-1}(L(\tilde{\nu}^{(n-1)})).\  
\end{equation*}
Therefore the map  $L$ has  the property that, for any $\tilde{\nu}$, $L(\tilde{\nu})$ is 
a cyclic permutation of coordinates of $L(\nu)$.
So, if $\nu$ is such that $L^j(\nu)=K$ for all $j=1,2,\ldots, n$, then $L^j(\tilde{\nu})=K$ for all $j$. Thus,  for every $\nu\in \mcm\n$ we have $L(\tilde{\nu})=L(\nu)$ and the following proposition holds.

\begin{pro}
If $\nu \in {\mathcal M}\n$, then $\tilde{\nu} \in {\mathcal M}\n$.
\end{pro}

By  analogy with the two-dimensional case we define the Sibirsky ideal of system \eqref{nSYS} as follows. 
\begin{defin}
The \emph{Sibirsky ideal} of system \eqref{nSYS} is the ideal 
$$
I^{(n)}_S=\la  [{\nu}]-[\hat  \nu]\ : \ \nu \in \mcm\n  \ra,
$$
where 
\be \label{hat_nu}
\hat \nu =\nu{\widetilde{\bf T}}.
\ee
\end{defin}

Since $\hat \nu {\bf T}=\nu$ and 
$$
[\nu]-[\tilde \nu^{(2)}]= [\nu]-[\tilde \nu^{(1)}]+ [\nu^{(1)}]-[\tilde \nu^{(2)}] \in I\n_S,
$$
we see that  the Sibirski ideal is the same as the ideal  
\be 
\la  [{\nu}]-[\tilde  \nu^{(k)}]\ : \ \nu \in \mcm\n ,  \ 
k=1,\dots, n-1 \ra.
\ee

To find  the Hilbert basis of $\mcm\n$ we use the following statement.
\begin{pro}\label{pro_M}
 Let $\nu\in {\mathbb N}_0^{n\ell}$. Then $\nu \in \mcm\n $ if and only if $\nu$ satisfies 
 the system
 \be \label{eqM2}
 \begin{aligned}
 &L^1(\nu)-L^2(\nu)=L^2(\nu)-L^3(\nu)=\ldots=L^{n-1}(\nu)-L^n(\nu)=0.\\
 \end{aligned}
 \ee 
\end{pro}
\begin{proof}
By definition of $\mcm\n$, it is obvious that every $\nu \in \mcm\n$ satisfies \eqref{eqM2}.

To prove the other implication, let $\nu\in {\mathbb N}_0^{n\ell}$ satisfy \eqref{eqM2}. It follows that  $L^1(\nu)=L^2(\nu)=\ldots=L^n(\nu)=K$  holds. Further, as in the proof of Theorem \ref{THMinvarM}, the  equality 
$L^1(\nu)+L^2(\nu)+\ldots +L^n(\nu)=nK$ implies that $K$ is a nonnegative integer. 
\end{proof}

Like in the two-dimensional case 
the  Hilbert basis of the monoid $\mcm\n$ can be computed using Algorithm \ref{alg1}.

Let $\mathfrak{M}\n$ be 
the $(n-1)\times n\ell$ dimensional matrix of system \eqref{eqM2},
\be \label{Mn}
\mathfrak{M}\n=\begin{bmatrix} P^{(1)} \ P^{(2)}\  \ldots \ P^{(\ell)}\end{bmatrix}, 
\ee
where every $P^{(k)}$ represents an $(n-1)\times n$ matrix of the form
\be \label{Pmat}
\begin{aligned}
    P^{(k)}=\begin{bmatrix} (p_1^{(k)}- p_2^{(k)}) & (p_n^{(k)}- p_1^{(k)}) & \ldots & (p_2^{(k)}- p_3^{(k)})\\
   (p_2^{(k)}- p_3^{(k)}) & (p_1^{(k)}- p_2^{(k)}) & \ldots & (p_3^{(k)}- p_{4}^{(k)})\\
     \vdots &   \vdots & \ddots & \vdots \\
   (p_{n-1}^{(k)}- p_n^{(k)}) & (p_{n-2}^{(k)}-p_{n-1}^{(k)}) & \ldots & (p_n^{(k)}-p_1^{(k)}) \end{bmatrix}.
 \end{aligned}
 \ee
 
For ${\bf y}=(y_1, \dots, y_{n\ell})$,  we denote  by  ${\bf y}^{\nu}$  the monomial
 $y_1^{\nu_1}y_2^{\nu_2}\ldots y_{n\ell}^{\nu_{n\ell}}$.
\begin{thm} \label{th_B3}
Let $Y\n=\left\langle [\nu] - {\bf y}^{\nu}\ |\ \nu\in H\n \right\rangle$ { be the ideal 
constructed from the \gb $\mathcal{G}_{\mathfrak{M}^{\n}}$} returned by Algorithm  \ref{alg1}  for the matrix $\mathfrak{M}\n$ 
(in particular, $H\n $ is the  Hilbert basis  of $\mathcal M\n$). 
Then ${ G\n}=\{[\nu]-[\hat{\nu}]\ |\ \nu\in H\n, \nu\neq \hat{\nu}\}$ is a Gr\"obner basis of $I\n_S$.
\end{thm}
\begin{proof}
Let $ [\mu]-[\hat{\mu}]\neq 0 $ ($ [\mu]-[\hat{\mu}]\not \in G\n   $) be an element of the generating set of the ideal  $ I\n_S$. 
We show that $[\mu]- [\hat{\mu}]\in \la  G\n \ra.$
Since  $\mu \in \mcm\n$, $H\n$ is the Hilbert basis of $\mcm\n$ 
and  $\mu\neq {\hat \mu}$, 
there exists $\nu\in H\n$ such that $\nu\mid \mu$.  Let $\sigma_1=\mu-\nu.$   
Since relation \eqref{divmn} still remains valid, 
repeating the reasoning of Theorem \ref{th_B}, 
we conclude that   $G\n$ is a \gb  basis for $I\n_S$. 
\end{proof}

Deleting from $G\n$ the redundant binomials,  we obtain the  reduced \gb
of $I\n_S$ denoted by $G\n_R$.
Adding the selfcongugate monomials to the monomials of $G\n_R$ we obtain 
the  Hilbert basis of $\mathcal{M}{\n}$.

\subsection{Invariants and normal forms} \label{sec_INF}

Let $\mathcal{V}^n$ be the space of vector fields $v: \mathbb{C}^n \rightarrow \mathbb{C}^n$, where the components of $v$ are power series in $x_1, \ldots, x_n$ that vanish at the origin. These series can be either convergent or formal.

Let $\vnj \subseteq \mathcal{V}^n$ be the space of polynomial vector fields $v: \mathbb{C}^n \rightarrow \mathbb{C}^n$, with each component $v_i$ (for $i = 1, \ldots, n$) being a homogeneous polynomial of degree $j$. Note that both $\mathcal{V}^n$ and $\vnj$ are vector spaces over $\mathbb{C}$, and $\mathcal{V}^n=\bigoplus_{j=0}^\infty\vnj$. 
  

It is clear  that  any   formal invertible  change of coordinates of the form 
		\be
		\label{yhy}
		\xb = \yb+ \  H(\yb)= \yb +\sum_{j=2}^\infty  H_j(\yb), 
		\ee
		with  $ H_{j} \in 
		\mathcal{V}^n_{j}$ for all $j\geq 2$, 
		brings system (vector field) \eqref{Asn} to a system of a similar form, 
		\be
		\label{linearni}
		\dot \yb = A_s \yb+ G(\yb),
		\ee
		where 
		$$ G(\yb)= \sum_{j=2}^\infty \ G_j(\yb), \quad \mbox{ and }  G_j(\yb)\in \vnj \quad \mbox{ for all } j \geq 2.
		$$ 
Let $\lambda_1, \ldots, \lambda_n$ be the eigenvalues of $A_s$, define $\lambda=(\lambda_1, \ldots, \lambda_n)^\top$ and $|\alpha|=\alpha_1+\alpha_2+\dots+\alpha_{n}$ for $\alpha=(\alpha_1, \ldots, \alpha_n)\in \mathbb{N}^n_0$. Consider the term $u\xb^\alpha$ (where $|\alpha|>1$) in ${\bf e}_k G(\yb)$ or ${\bf e}_k X(\xb)$ (here, ${\bf e}_k$ is the $n$-dimensional row unit vector for $k=1,\ldots, n$, and $X(\xb)$ represents the nonlinear part of \eqref{Asn}). This term is called \emph{resonant} if $k$ and $\alpha$ satisfy
		\be \label{res_c_e} 
		\lambda \cdot \alpha - \lambda_k=0.
		\ee

	

System \eqref{linearni} is said to be in the \emph{Poincar\'e--Dulac normal form}, or simply in the normal form, if $G(\yb)$ includes only resonant terms.

According to the Poincar\'e--Dulac theorem, system \eqref{Asn} can be transformed into a normal form using the substitution given in \eqref{yhy}. This transformation is  referred to as a  normalization or a normalizing transformation. 
In general, the normalization of system \eqref{linearni} is not unique.

For any $ m  \in  \mathrm{im}L\subseteq \Z^{n}  $, where $L$ is defined in \eqref{defL}, a  (Laurent)  polynomial $p(a)$, where
			$p(a) = \sum_{\nu\in \Supp(p)}p^{(\nu)}[\nu]$, is an $ m$-\emph{polynomial} 
			if $L(\nu) =  m$ for every  $\nu \in \Supp(p) \subset \N^{n\ell}_{0}$. 
			
		
		For any $m\in\Z^n$, 
		let $ R_{ m}$  be the subset of $\Q[a]$ consisting of all $ m $-polynomials. Denote by   $R$  the direct sum of   $ R_{ m}$,
		$$
		R=\bigoplus_{ m\in \Z^n }R_{ m}. 
		$$
		Since 
		$$
		R_{ m_1}R_{ m_2}\subseteq R_{ m_1+ m_2},
		$$
$R$ is a $\Z^n-$graded ring. 

Let  $\bK  =(K,K, \dots, K)$ and $ \yb^{\bK}=
y_1^K y_2^K\cdots y_n^K.$ The following 
theorem gives the structure of the coefficients 
of normal forms of systems \eqref{nSYS}.
\begin{thm} \label{Thm_nf}
There is a normal form \eqref{linearni} of system \eqref{nSYS} such that 
$$
 G_j(\yb)= \yb^{\bK} (\yb \odot 
q(a)),
$$
where $\bK  =(K,K, \dots, K)$, 
$q(a)=(q_1(a), \dots, q_n(a))^\top$ and
$$
q_i(a)\in R_{\bK} \ {\ \ for \ }  i=1,\dots, n.
$$
\end{thm}
\begin{proof}
System \eqref{nSYS}
is a particular case of system 
(47) of \cite{PR}, which is 
\be \label{sys_k}  
		\dot x_k=  \lambda_k x_k+ x_k \sum_{{\bi} \in \Omega_k}   a^{(k)}_{\bi} \xb^{\bi}, \quad   k=1,\dots, n,
		\ee
		where $\Omega_k$ is a  set of $n$-tuples $\bi=(i_1,\dots,i_n)$ whose $k-$th entry is from $\N_{-1}=\{-1\} \cup \N_0$ and all other entries are from $\N_0$. 

The term   $x_k a^{(k)}_{\bi} \xb^{\bi}$ in  
the $k$-th equation of \eqref{sys_k} is resonant   
if 
\be \label{res_i}
(\bi +{\bf e}_k)\cdot \lambda -\lambda_k=
\bi \cdot \lambda=0. 
\ee
By  Theorem 1 of \cite{PR}, there is a normal 
form of \eqref{sys_k} such that if $ p(a)  y_k \yb^{\nu}$ is a term in  
the $k$-th equation of the normal  form,
then $p(a)$ is $L(\nu)$-polynomial (that is, $p(a)\in R_{L(\nu)}$). 

In case of system \eqref{nSYS}, condition \eqref{res_i} is equivalent to
\be \label{eq_Iz}
\bi\cdot  \bar \zeta=0.
\ee
It is shown in the proof of Theorem \ref{THMinvarM} that all nonnegative solutions
$\bi=(i_1,i_2,\dots. i_n)$ to \eqref{eq_Iz}
are of the form $i_1=i_2=\dots=i_n=i$. 
Thus, by  Theorem 1 of \cite{PR}, we conclude that  the present theorem holds true. 
\end{proof} 
 
\begin{remark}
 Note that in the present paper $L(\nu)$ is 
 a column vector, whereas in \cite{PR} it is a 
 row vector. 
\end{remark}

Denote by $\Q[H\n ]$ the polynomial
subalgebra generated by the set of monomials $$\{ [\nu] : \nu \in \mathcal{M}\n \}.$$ 
From  Theorem   \ref{Thm_nf}, we have the following. 
\begin{cor} \label{cor_inv}
The polynomials $q_i(a)$
of the statement of Theorem \ref{Thm_nf}  are elements of $\Q[H\n ]$.
\end{cor}
{Thus, we have the important property  that the coefficients  $q_i(a)$ ($i=1,\dots,n$) of the normal form  are    polynomials with  rational coefficients in the invariants of group \eqref{ch3d}.
This property established earlier in the two-dimensional case in \cite{LiuLi} has been  proven to be very useful in the investigations of bifurcations of limit cycles and critical periods in the two-dimensional case (see e.g. \cite{FLRS,LLR,LPR}).
}

\begin{remark} 
		For a vector field  $ v(\xb)\in \mathcal{V}^n $ of the form \eqref{Asn}, we denote by $Dv(\xb)$ the $n\times n$ matrix of partial derivatives of $v(\xb)$. 
		The linear operator $\pounds : \mathcal{V}^n\to \mathcal{V}^n$  defined by 
		\be \label{ho}
		(\pounds v)(\xb) =Dv(\xb) A_sx - A_s v(\xb)
		\ee
		is called the \emph{homological operator}. 
  
It is clear that system \eqref{Asn} is in the Poincar\'e-Dulac normal form if $v(\xb) \in \ker \pounds. $
Since 
the algebra of polynomial first integrals 
of \eqref{Zlin} is generated by the single monomial \eqref{I_int},
the Stanley decomposition \cite{Mur,Stan,SW} of the normal form module in $\mathcal{V}^n$ is
$$
\ker \pounds =\oplus_{i=1}^n \C[[\Phi]] x_i {\bf e}_i,
$$
where $\Phi$ is the monomial \eqref{I_int}.

It follows from Theorem \ref{Thm_nf} that the  decomposition of the normal form module can be written as 
\be \label{Stan_dec}
 \ker \pounds =\oplus_{i=1}^n  \Q[H\n ][[\Phi]] x_i {\bf e}_i.
\ee 
\end{remark}


\subsection{Reversibility  and equivariance  in higher dimensional systems}
 
Unlike in the two-dimensional case,  $G\n_R$ does not have a property similar to the one in Proposition \ref{pro_m}. Consequently,  the variety of $I\n_S$ is not the  Zariski closure of the set of all equivariant systems in family \eqref{nSYS} and the ideal  $I\n_S$ is not a lattice ideal. 

In this section, we study the relationship between the ideal $I\n_S$
and the ideals defining the sets of $\zeta$-reversible and 
equivariant systems in family \eqref{nSYS}.
 
\begin{pro}
System     \eqref{nSYS} is $\zeta$-reversible if 
\be\label{zRev0}
\zeta R(Ax)=\widetilde T_n R(x)
\ee
for some matrix $A$ of the form  \eqref{Amatrix},  and it  is  equivariant  if 
\be\label{eRev0}
 R(Ax)=\widetilde T_n R(x)
\ee
for some matrix $A$ of the form  \eqref{Amatrix}. 
\end{pro}
\begin{proof}
By Definition \ref{defZrev},
system \eqref{nSYS} is  $\zeta$-reversible 
if 
\be \label{Zrev1}
 \zeta {\mc Z}(Ax) +  \zeta (A x) \odot R(Ax)= A 
{\mc Z}x+ A (x \odot R(x)).
\ee

By \eqref{ABT}, we can  write  
$A={\mc B}\widetilde T_n {\mc B}^{-1}. $
Since 
$$
 \zeta {\mc Z}A  - A 
{\mc Z}  =   \zeta {\mc Z} {\mc B}\widetilde T_n {\mc B}^{-1}   - {\mc B}\widetilde  T_n {\mc B}^{-1} 
{\mc Z}={\mc B}\left( \zeta {\mc Z}\widetilde T_n  - \widetilde T_n  
{\mc Z}\right){\mc B}^{-1}=0,
$$
we see  that \eqref{Zrev1} is equivalent to 
\be \label{zRev}
\zeta (A x) \odot R(Ax)=A (x \odot R(x)).
\ee
Observing that 
$$
A (x \odot R(x)) = Ax \odot (\widetilde T_n R(x) ),
$$
we conclude that   the latter equality  is equivalent to \eqref{zRev0}.

The correctness of \eqref{eRev0} is established similarly. 
\end{proof} 

Recall that, for ${\bf p}=(p_1,\dots, p_n)^\top$ and $\alpha=(\alpha_1,\dots, \alpha_n)$,
we denote 
$$
\alpha^{\bf p}:=\alpha_1 ^{p_1} \alpha_2 ^{p_2} \dots \alpha_n^{p_n }.
$$
Since  
$$
(A {\bf x})^{\bf p} =\alpha^{\bf p} {\bf x}^{T_n({\bf p})}, 
$$
 we see  from \eqref{zRev0} that the condition of $\zeta$-reversibility of system \eqref{nSYS} can be written as 
 $$
\zeta a^{(i)}_{T^{i-1}(\pk)} {\bf \alpha}^{T_n^{i-1}(\pk)}= a^{(i+1)}_{T_n^i(\pk)} \ \  \text{for\ } i=1, \dots, n-1, \ \     a^{(n)}_{T_n^{n-1}(\pk)}= a^{(1)}_{\pk}, \ \  
 $$
for all $k=1,\dots , \ell$; or in a short form, again using the  notation $]m[_n$ for the remainder of integer $m$ when divided by $n$, as
\be\label{cond_rev}
\zeta a^{(i)}_{T^{i-1}_n(\pk)} {\bf \alpha}^{T_n^{i-1}(\pk)}= a^{(]i[_n+1)}_{T^i_n(\pk)}\,,\quad i=1, \dots, n,  \mbox { and } k=1,\dots , \ell. 
\ee
Setting in \eqref{cond_rev} $\zeta =1$, we obtain the conditions for equivariance of system \eqref{nSYS}. 

\begin{pro}\label{pro48}
1) The Zariski closure of the set of $\zeta$-reversible systems in the space of parameters of system \eqref{nSYS}  is the variety of the ideal
\be \label{Izeta}
I_\zeta=I^{(\zeta)}\cap \C[a],
\ee
where 
\begin{multline*}
I^{(\zeta)}=\left\langle 1-\alpha_1 \alpha_2 \cdots \alpha_n ,
\zeta a^{(i)}_{T^{i-1}_n(\pk)} {\bf \alpha}^{T_n^{i-1}(\pk)} - a^{(]i[_n +1)}_{T^i_n(\pk)} : \ i=1, \dots, n, \ k=1,\dots , \ell 
\right\rangle
  \end{multline*}
is an ideal 
in the ring  ${\mathbb C}[ \alpha^\pm,a]:=
\C[\alpha_1,\dots, \alpha_n, \alpha_1^{-1},\dots, \alpha_n^{-1}, a]$ of Laurent polynomials.
\\
2)
The  Zariski closure of the set of equivariant  systems
in  the space of parameters of system \eqref{nSYS}
is the variety of the ideal
\be\label{I_E}
\mathcal{I}_E=I^{(E)}\cap \C[a],
\ee
where 
\begin{multline*}
I^{(E)}=\left\langle 1-\alpha_1 \alpha_2 \cdots \alpha_n ,
a^{(i)}_{T^{i-1}_n(\pk)} {\bf \alpha}^{T_n^{i-1}(\pk)} - a^{(]i[_n +1)}_{T^i_n(\pk)} :  \ i=1, \dots, n, \ k=1,\dots , \ell
\right \rangle.
  \end{multline*}
 \\
  {
3)  The ideals ${\mc I}_E $ and $ I_\zeta$ are prime.}
\end{pro}
{
\begin{proof}
The correctness of statements 1 and 2 
follows from the Closure  Theorem  (see e.g. \cite{Cox}). 

To prove 3), the same argument as in Subsection \ref{sub32} can be used to show that ${\mc I}_E$ and $ I_\zeta$ are prime ideals after considering the ideal in  $\C[\alpha_1,\dots, \alpha_n, \alpha_1^{-1},\dots, \alpha_n^{-1}, a, u_1, \dots, u_\ell ]$ that is generated by the set of all generators of $I^{(\zeta)}$ in addition to  
$$
a^{(1)}_{{\bf p}^{(1)}}- u_1, \dots, a^{(1)}_{{\bf p}^{(\ell)}}- u_\ell.
$$

\end{proof}
}

Note that the ideals $ I^{(\zeta)}$ and $ I^{(E)}$
defined above are ideals in the Laurent polynomial ring. However, we can treat them as ideals in $\C[\alpha, a]$ after multiplying their polynomial generators by a suitable monomial in variables $\alpha_1, \dots, \alpha_n$.

\begin{thm}
 If the parameters of system \eqref{nSYS} belong to the variety of the ideal $I_\zeta$, then the corresponding system has a  first integral of the form \eqref{Psi_int}.
\end{thm}
\begin{proof}
 By Proposition \ref{prop25},   $\zeta$-reversible systems \eqref{nSYS} have a  first  integral of the form  \eqref{Psi_int}.
 
 According to Proposition  8  of \cite{LPW}   the set  in the space of parameters corresponding to  system
 with such integral is an algebraic set. Taking into account the previous proposition we conclude that the statement of the theorem holds.
\end{proof}
   

Recall that  $|\nu|=\nu_1+\nu_2+\dots+\nu_{n\ell}$ for any $\nu\in {\mathcal M}\n$.

\begin{pro}\label{pro_zeta_rev}
If system \eqref{nSYS} is $\zeta$-reversible, then, for any $\nu\in {\mathcal M}\n$, it holds that 
\be\label{71}
\zeta^{\left|\nu\right|} [\nu]-[\hat{\nu}]=0,
\ee
Furthermore, if the system is  equivariant, then, for any $\nu\in {\mathcal M}\n$, 
\be 
\label{72}
[\nu]-[\hat{\nu}]=0.
\ee
\end{pro}

\begin{proof}
Let $\nu\in {\mathcal M}\n$.
Then by \eqref{hat_nu}, 
$\hat \nu = \nu \widetilde{{\bf T}}$
and using \eqref{cond_rev}, we have 
$$
[\hat \nu]= \zeta^{|\nu|} [\nu]
\alpha^ { \sum_{j=1}^n \sum_{k=1}^{\ell} \left(T_n^j(\widetilde{\bf{p}}^{(k)})\cdot \nu^{\{k\}} \right)\,{\bf e}_j
},
$$
where ${\bf e}_j$ is the $j$-th unit vector.
Taking into account \eqref{Lmap}, \eqref{defL}, \eqref{Al_p}
and that $L(\nu)=(K_0,\dots, K_0)$ for a $K_0\in \N_{0}$, we 
obtain 
$$
[\hat \nu]= \zeta^{|\nu|} [\nu] \alpha^{L(\nu)}=
 \zeta^{|\nu|} [\nu] \alpha_1^{K_0}\cdots \alpha_n^{K_0}=  \zeta^{|\nu|} [\nu] (\alpha_1 \alpha_2\cdots \alpha_n)^{K_0} =\zeta^{|\nu|} [\nu].
$$
Therefore \eqref{71} holds. The correctness of \eqref{72} can be verified similarly. 
\end{proof}

\begin{lemma}  
If $\nu =\hat \nu $ 
then   $\zeta^{\left| \nu \right|}[\nu]-[\hat{\nu}]=0$.
\end{lemma}
\begin{proof}
If  $\hat \nu=\nu$, then $\nu_{(k-1)n+1}=   \nu_{(k-1)n+2}=\dots   =\nu_{kn }$ for $k=1,\dots,\ell$. Therefore $|\nu|=n m $ for some $m \in \N$ and so $\zeta^{|\nu|}=1$.
\end{proof}


The following theorem gives another characterization 
of the sets of $\zeta$-reversible and  equivariant systems, which allows to compute the set looking for the kernel of a polynomial map.

\begin{thm} \label{pro412}
Let $\zeta$ be a primitive $n$-th root of unity.
System \eqref{nSYS} is $\zeta$-reversible  if and only if for every 
$\pk \in S^{(n)}$,   $k=1,2,\ldots,\ell$, and every $j=1,\dots, n$,
\begin{eqnarray}\label{T3-equiv}
\begin{aligned}
 a^{(j)}_{T^{j-1}_n(\pk)} 
 =\zeta^j 
 y_{k} {\bf t}^{ T^{j-1}_n(\pk) }
\end{aligned}    
\end{eqnarray}
for some $y_k$ and ${\bf t}=(t_1, \dots, t_n)$, where 
\be
\label{t_prod} t_1t_2\cdots t_n=1.
\ee

If \eqref{T3-equiv} holds with $\zeta=1,$ then system \eqref{nSYS} is equivariant.  
\end{thm}

\begin{proof}
It is enough to show that the system
$$ 
\zeta a^{(i)}_{T^{i-1}_n(\pk)} {\bf \alpha}^{T_n^{i-1}(\pk)} = a^{(]i[_n +1)}_{T^i_n(\pk)} \quad  \text{ where }\   \ i=1, \dots, n, \text{ and } \ k=1,\dots , \ell
$$
is equivalent to \eqref{T3-equiv}. First notice that \eqref{T3-equiv} is equivalent to the system
\begin{eqnarray} \label{equiv}
\begin{aligned}
&y_k=\zeta^{-j} a^{(j)}_{T^{j-1}_n(\pk)} {\bf t}^{ T^{j-1}_n(-\pk) } \ \text{ where }\ j=1,\ldots,n\ \text{ and }\ k=1,\ldots,\ell.
\end{aligned}    
\end{eqnarray}
 
Let $\overline{n}=(0,1,\ldots,n-1)$ and define ${\bf \alpha}^{(K_1,K_2,\ldots,K_n)}:=\alpha_1^{K_1} \alpha_2^{K_2} \cdots \alpha_n^{K_n}$ for any $n$-tuple of nonnegative integers $(K_1,K_2,\ldots,K_n)$. Consider the  substitutions given on the $n$-th powers of $t_j$'s by 
\begin{eqnarray}\label{tprod}
\begin{aligned}
&t_1^n\ =\ \alpha_2\,\alpha_3^2\,\alpha_4^3\cdots \alpha_n^{n-1}\ =\ {\bf \alpha}^{\overline{n}}\\
&t_2^n\ =\ \alpha_1^{n-1}\,\alpha_3\,\alpha_4^2\cdots \alpha_n^{n-2}\ =\  {\bf \alpha}^{T_n\left(\overline{n}\right)}\\
&t_3^n\ =\ \alpha_1^{n-2}\,\alpha_2^{n-1}\,\alpha_4\cdots \alpha_n^{n-3}\ =\ {\bf \alpha}^{T_n^2\left(\overline{n}\right)}\\
&\ \ \ \ \ \vdots\\
&t_n^n\ =\ \alpha_1\,\alpha_2^2\,\alpha_3^3\cdots \alpha_{n-1}^{n-1}\ = \ {\bf \alpha}^{T_n^{n-1}\left(\overline{n}\right)}
\end{aligned}    
\end{eqnarray}
or,  equivalently,
$$
t_j^n={\bf \alpha}^{T_n^{j-1}\left(\overline{n}\right)} \ \text{ for all }\ j=1,2,\ldots,n.
$$

Using these substitutions and  the notation 
$$
\widehat{{\bf{\alpha}}}=\left({\bf \alpha}^{\overline{n}}
,{\bf \alpha}^{T_n(\overline{n})},\ldots,{\bf \alpha}^{T_n^{n-1}(\overline{n})}\right)
$$
system \eqref{equiv} is transformed  to
\begin{eqnarray*}
\begin{aligned}
&y_k=\zeta^{-j} a^{(j)}_{T^{j-1}_n(\pk)} {\widehat{\bf{\alpha}}}^{\frac{1}{n} T^{j-1}_n(-\pk) }\ \text{ where }\ j=1,\ldots,n\ \text{ and }\ k=1,\ldots,\ell.
\end{aligned}    
\end{eqnarray*}

Notice that for arbitrary consecutive $j,j+1$, where   $j=1,2,\ldots,n-1$,  and an  arbitrary $k$, we have
$$\zeta 
a^{(j)}_{T^{j-1}_n(\pk)} {\widehat{\bf{\alpha}}}^{\frac{1}{n} T^{j-1}_n(-\pk) }=a^{(j+1)}_{T^{j}_n(\pk)} {\widehat{\bf{\alpha}}}^{\frac{1}{n} T^{j}_n(-\pk) }
$$
or, equivalently,
$$\zeta 
a^{(j)}_{T^{j-1}_n(\pk)} {\widehat{\bf{\alpha}}}^{\frac{1}{n}\left( T^{j-1}_n(-\pk) +T^{j}_n(\pk)\right)}=a^{(j+1)}_{T^{j}_n(\pk)}\ \text{ for all } j=1,\ldots,n-1, 
$$
and
$$\zeta 
a^{(n)}_{T^{n-1}_n(\pk)} {\widehat{\bf{\alpha}}}^{\frac{1}{n}\left( T^{n-1}_n(-\pk) +\pk \right)}=a^{(1)}_{\pk}.
$$

Taking into account equalities \eqref{cond_rev}, which present the condition of $\zeta$-reversibility  of system \eqref{nSYS}, it remains to prove 
$$
{\widehat{\bf{\alpha}}}^{\frac{1}{n}\left( T^{j-1}_n(-\pk) +T^{j}_n(\pk)\right)}= {\bf \alpha}^{T_n^{j-1}(\pk)}\ \text{ for all }\ j=1,\ldots, n-1
$$
and 
$${\widehat{\bf{\alpha}}}^{\frac{1}{n}\left( T^{n-1}_n(-\pk) +\pk\right)}= {\bf \alpha}^{T_n^{n-1}(\pk)}.
$$

To this end,  observe that if $T^{j-1}_n(\pk)=\left( q_1,q_2,\ldots,q_n\right)$, then $$T^{j}_n(\pk)=\left(q_n,q_1,\ldots,q_{n-1}\right)$$ and, therefore,  $$T^{j-1}_n(-\pk) +T^{j}_n(\pk)=\left(q_n-q_1,q_1-q_2,\ldots,q_{n-1}-q_n\right).$$
In view of \eqref{Al_p} the  calculations give 
\begin{eqnarray*}
\begin{aligned}
& {\widehat{\bf{\alpha}}}^{\frac{1}{n}\left( T^{j-1}_n(-\pk) +T^{j}_n(\pk)\right)}
=\left(\left({\bf \alpha}^{\overline{n}}\right)^{q_n-q_1}
\left({\bf \alpha}^{T_n(\overline{n})}\right)^{q_1-q_2}\cdots\left({\bf \alpha}^{T_n^{n-1}(\overline{n})}\right)^{q_{n-1}-q_n}\right)^{\frac{1}{n}}\\
& =\left( \left(\alpha_2\alpha_3^2\cdots \alpha_n^{n-1}\right)^{q_n-q_1}\left(\alpha_1^{n-1}\alpha_3\cdots \alpha_n^{n-2}\right)^{q_1-q_2} \cdots\left(\alpha_1\alpha_2^2\cdots \alpha_{n-1}^{n-1}\right)^{q_{n-1}-q_n}\right)^{\frac{1}{n}}\\
& =\left( \alpha_1^{nq_1}\,\cdots\, \alpha_n^{nq_n}(\alpha_1\,\cdots\, \alpha_n)^{-(q_1+\ldots+q_n)}\right)^{\frac{1}{n}}\\
& =\alpha_1^{q_1}\,\cdots\, \alpha_n^{q_n}\\
& ={\bf \alpha}^{T_n^{j-1}(\pk)}.
\end{aligned}    
\end{eqnarray*}
Similarly, we see that  ${\widehat{\bf{\alpha}}}^{\frac{1}{n}\left( T^{n-1}_n(-\pk) +\pk\right)}= {\bf \alpha}^{T_n^{n-1}(\pk)}$ holds.

Since 
$$\overline{n} + T_n(\overline{n})+ \ldots+T_n^{n-1}(\overline{n})= 
\left(\frac{(n-1) n}{2}, \dots, \frac{(n-1) n}{2}\right),
$$  
taking into account \eqref{Al_p} and \eqref{tprod}
we see that \eqref{t_prod} takes place.
\end{proof}


Let $\mathfrak{A}\n$ be 
the $(n-1)\times n\ell$ dimensional matrix 
\be \label{An}
\mathfrak{A}\n=\begin{bmatrix} 
P^{(1)} \ P^{(2)}\  \cdots \ P^{(\ell)} \\
O^{(1)} \ O^{(2)}\  \cdots \ O^{(\ell)}
\end{bmatrix}, 
\ee
 where every $O^{(k)}$ represents an $\ell \times n$ matrix whose entries are zero except the $k^{th}$-row entries, which are all one,
 and $P^{(k)}$ is defined by \eqref{Pmat}.
 

Let
  \be \label{Lat_L}
  \msl = \{\beta  \in \Z^{n\ell} \, : \, \mathfrak{A}\n  \beta  = 0 \}
  \ee 
  be the  lattice in $\Z^{n\ell}$ and $I_\msl$ be  its lattice ideal. 
The  toric ideal  defined by  matrix  \eqref{An} is denoted by 
$  I_{\mathfrak{A}\n}$.

Let ${\bf s}=(s_1,\dots, s_{n-1}, (s_1 s_2\cdots s_{n-1})^{-1})$.
By  \eqref{t_prod}, 
equations \eqref{T3-equiv}
are  equivalent to 
\be \label{T3-equiv_red}
a^{(j)}_{T^{j-1}_n(\pk)} 
 =
\zeta^j y_{k} {\bf s}^{ T^{j-1}_n(\pk) }, \quad   k=1,\dots \ell, \ \ j=1,\dots, n.
\ee

For $k=1,\dots, \ell$, let 
\be \label{PmatP}
\begin{aligned}
  \widehat    P^{(k)}=\begin{bmatrix} (p_1^{(k)}- p_n^{(k)}) & (p_n^{(k)}- p_{n-1}^{(k)}) & \ldots & (p_2^{(k)}- p_1^{(k)})\\
   (p_2^{(k)}- p_n^{(k)}) & (p_1^{(k)}- p_{n-1}^{(k)}) & \ldots & (p_3^{(k)}- p_{1}^{(k)})\\
     \vdots &   \vdots & \ddots & \vdots \\
   (p_{n-1}^{(k)}- p_n^{(k)}) & (p_{n-2}^{(k)}-p_{n-1}^{(k)}) & \ldots & (p_n^{(k)}-p_1^{(k)}) \end{bmatrix}.
 \end{aligned}
 \ee
Considering equations \eqref{T3-equiv_red} as a polynomial mapping in the ring of Laurent polynomials, 
we see that the corresponding toric ideal  is defined by the matrix 
\be \label{Anh}
\widehat {\mathfrak{A}}\n=\begin{bmatrix} 
\widehat {P}^{(1)} \ \widehat{P}^{(2)}\  \cdots \ \widehat{P}^{(\ell)} \\
O^{(1)} \ O^{(2)}\  \cdots \ O^{(\ell)}
\end{bmatrix}. 
 \ee
 
It is not difficult to verify that the lattice 
$$
 \{\beta  \in \Z^{n\ell} \, : \, \widehat{\mathfrak{A}}\n \beta  = 0 \}$$
is the same as lattice \eqref{Lat_L}. 

\begin{pro} \label{EI}
  $\mathcal{I}_E = I_{\mathfrak{A}\n} =  I_{ \widehat{\mathfrak{A}}\n}   .$
\end{pro}
{
\begin{proof}  
Since  the lattices defined by matrices \eqref{An} and \eqref{Anh} are the same, the ideals  $I_{\mathfrak{A}\n} $ and $ I_{ \widehat{\mathfrak{A}}\n}$ are the same.

The ideal 
$I_{\mathfrak{A}\n}$ is a toric ideal, hence  it is prime. The ideal $\mathcal{I}_{E}$ is prime by Proposition \ref{pro48}. Furthermore, by Proposition \ref{pro48} and Theorem \ref{pro412},
the varieties of $I_{\mathfrak{A}\n}$ and  $\mathcal{I}_{E}$ coincide. Therefore the two prime ideals coincide as well.
\end{proof}
}

For $\beta \in \Z^{n\ell}$, 
let $\beta^+$
and $\beta^-$ be $n\ell$-tuples  in  $\N_0^{n\ell}$ 
with entries 
\be \label{beta_def}
\beta_i^+=\begin{cases}  \beta_i {\rm \ if \ \beta_i\ge 0 }  \\   0 \  {\rm \ if \ \beta_i < 0 } \end{cases}, \qquad  \beta_i^-=\begin{cases}  \phantom{-}  0  {\rm \ if \ \beta_i> 0 }  \\   -\beta_i  {\rm \ if \ \beta_i \le  0 }  \end{cases}.
\ee

From  Proposition \ref{EI} and Proposition 3.16 of \cite{HHO}  we see that
  $$
  \mathcal{I}_E = I_{\mathfrak{A}\n}   = \left\langle [\beta^+] -[\beta ^-] \, : \, \mathfrak{A}\n \beta  = 0 \right\rangle = I_\msl$$ 
and, therefore,  $\mathcal{I}_E$ is a lattice ideal.


Denote by $ \mathsf{a} $
the product of all parameters of system \eqref{nSYS},  that is, $${\mathsf a}= a^{(1)}_{{\bf p}^{(1)}}\cdots a^{(n)}_{T_n^{n-1}({\bf p}^{(\ell)})} $$ 
and recall that for an ideal $I\subseteq \C[a]$ 
and the multiplicative set $\{1, \mathsf{a}, \mathsf{a}^2, \dots \} $, 
  $ I: {\mathsf  a} ^\infty$ denotes  the saturation of $I$ by $\mathsf{a}$. 

The following theorem establishes the connection between the ideal  $\mathcal{I}_E$ defining the set of systems 
which are equivariant with respect to 
cyclic permutation group $T_n$ and the ideal  $ I^{(n)}_S$ defining the set of invariants of the Lie group,  which generator is  the diagonalization of $T_n$. 
\begin{thm}\label{con_satur}
   $ I^{(n)}_S: {\mathsf  a} ^\infty = \mathcal{I}_E$. 
\end{thm}
{
The proof of this theorem is   a special case of the proof of Theorem \ref{conj_2} related to $\zeta$-reversible systems, which is presented below.


Recall that 
by Proposition \ref{pro_zeta_rev}
if system \eqref{nSYS} is $\zeta$-reversible, then for any $\nu\in {\mathcal M}\n$ it holds 
$$
\zeta^{\left|\nu\right|} [\nu]-[\hat{\nu}]=0.
$$
Denote 
$$
I\n_R=\la   \zeta^{\left|\nu\right|} [\nu]-[\hat{\nu}]\ :   \nu\in {\mathcal M}\n \ra.  
$$

The following statement establishes a relation of 
$ I\n_R$ and the ideal $I_\zeta$ defining the set of $\zeta$-reversible systems in family \eqref{nSYS}.
\begin{thm}\label{conj_2}
$ I\n_R : \mathsf{a}^\infty= I_\zeta. $       
\end{thm}


In the proof of Theorem \ref{conj_2} we face a particular challenge.
Generally, and in this paper, a binomial ideal is defined as an ideal generated by binomials of the form $\xb^\alpha - \xb^\beta$.
However, some literature, such as \cite{ES}, defines the following general expression 
\be\label{xbin}
\xb^\alpha - u\xb^\beta, 
\ee
where $u$ is a constant from the base field, as binomials. This distinction is crucial because ideals generated by binomials of the form \eqref{xbin} can include monomials, highlighting a fundamental difference between the theories discussed in \cite{ES} and \cite{HHO}.
  
To prove Theorem  \ref{conj_2} we need to work with ideals generated  by 
binomials $[\nu] - \zeta^m   [\mu]$, where $m\in \N_0$ and  $\zeta$ is a primitive $n$-th root of unity. Next, we introduce a generalization of toric ideals. 

Define the $\Q(\zeta)$-algebra homomorphism $\pi$ by 
\be \label{pi}
\begin{aligned}
\pi : 
 a^{(j)}_{T^{j-1}_n(\pk)} 
 \mapsto  \zeta^j 
 y_{k} {\bf t}^{ T^{j-1}_n(\pk) }  \quad \text{ for all } \ k=1,2,\ldots,\ell \ \text{ and } \ j=1,\dots, n. 
\end{aligned}    
\ee
\begin{defin}
 The kernel of $\pi$ is called $\zeta$-toric ideal and denoted by    $I^{(\zeta)}_{\mathfrak{A}^{(n)}}.  $
\end{defin}

Clearly, if we set $\zeta=1$  in \eqref{pi}, then the  corresponding 
ideal is the usual toric ideal defined by matrix \eqref{An}.

\begin{pro}
$
 \mathcal{I}_\zeta = I^{(\zeta)}_{\mathfrak{A}\n}.
$
\end{pro}
\begin{proof}
The equality follows from the fact both ideals are prime and have the same corresponding variety.
\end{proof}

For $\mu \in \N_0^{n\ell} $, 
we define 
$
w(\mu)= \bar \zeta^{\mu^{\{1\}}+  \mu^{\{2\}}+ \dots + \mu^{\{\ell\}} },   
$
where 
$$
\mu^{\{k\}}=(\mu_{(k-1)n+1},\mu_{(k-1)n+2},\ldots ,\mu_{kn} ), \quad k=1,\dots, \ell.
$$
In the short form we can write
$$
w(\mu)= \zeta ^{{\overline n}\, \cdot \sum_{k=1}^\ell \mu^{\{k\} }},
$$
where $\overline{n}=(0,1,\dots, n-1)$.


\begin{lemma}\label{lem424}
For      $\mu, \theta \in \N_0^{n\ell} $, we have the following 
\begin{enumerate}
\item $w(\mu+\theta) = w(\mu) w(\theta)$
\item $w(\mu-\theta) =\frac{ w(\mu)} {w(\theta)}$
\item $w(\hat \mu)= \zeta^{|\mu|}w  (\mu)$. \label{item3}
\end{enumerate}
\end{lemma} 

\begin{proof}
The first two items can be easily verified by straightforward computations.
To prove \ref{item3}, let $ u=(u_1, \dots, u_n)$, where, for $i=1, \dots, n$,
    $$
    u_i=\sum_{k=1}^\ell \mu_{(k-1) n+i}.  
    $$
Then 
$
w(\mu)=  \zeta^{ \overline{n} \cdot u }. 
$
Observe that 
$
\hat \mu =(\hat  \mu^{\{1\}},   \hat  \mu^{\{2\}}, \dots, \hat 
\mu^{\{\ell\}}),
$
where 
$$
\hat  \mu^{\{k\}}=   (\mu_{kn}, \mu_{(k-1)n+1},\mu_{(k-1)n+2},\ldots ,\mu_{kn-1}),   \quad k=1,\dots, \ell.
$$
Therefore,  
$$
w(\hat \mu) = \zeta^{ 0 u_n+1 u_1+2 u_2+\dots +(n-1) u_{n-1} }. 
$$
Since $|\mu|=u_1+\dots +u_n$,  
$$
w( \mu)  \zeta^{|\mu  |}  = \zeta^{  1 u_1+2 u_2++\dots +(n-1) u_{n-1} +n u_n } =w(\hat \mu).
$$
\end{proof}

\begin{pro}\label{pro426}
    $\zt =\la w(\beta^- )  [\beta^+ ]-  w(\beta^+) [\beta^- ] \ :\ \mathfrak{A}^{\n} \beta =0 \ra.$ 
\end{pro}
\begin{proof}
It is easy to see that the binomial $ w(\beta^- )  [\beta^+ ]-  w(\beta^+) [\beta^- ] \in \ker \pi$ for all $\beta$ such that $\mathfrak{A}^{\n} \beta =0$. The opposite inclusion can be proved using a  similar reasoning as in the proof of Lemma 4.1 of \cite{St-GBCP}.
\end{proof}


The proof of the next statement is similar to the proof of  Theorem 3.20 of  \cite{HHO}.
\begin{thm}\label{thm427}
   $\zt : \mathsf{a}^\infty = \zt. $ 
\end{thm}
\begin{proof}
   We only need to prove that $\zt: \mathsf{a}^\infty \subset \zt. 
   $
Let 
\be \label{fin}
f\in \zt : \mathsf{a}^\infty .
\ee
Since  
by Corollary 1.7 of \cite{ES}      $\zt : \mathsf{a}^\infty $  is a binomial ideal, we can assume that 
$$
f=f_\beta =u [\beta^+ ]- v [\beta^-],
$$
where $\beta^+, \beta^-$ are some $n\ell$-tuples in $   \N_0^{n\ell}$ and $u, v \in \Q(\zeta).$
By \eqref{fin} we have 
\be \label{bin}
\mathsf{a}^k (u [\beta^+ ]- v [\beta^-])\in \zt  
\ee
for some $k\in \N_0 $.

If $supp(\beta^+)\cap supp(\beta^{-})\ne \emptyset $ we can write $\beta^+=\tilde \beta^++\gamma$, 
 $\beta^-=\tilde \beta^-+\gamma$, 
 where  $supp(\tilde \beta^+)\cap supp(\tilde \beta^-)= \emptyset $. 
Thus, we can write \eqref{bin} as 
$$ \mathsf{b}= \mathsf{a}^k [\gamma] (u [\tilde \beta^+ ]- v [\tilde \beta^-])\in \zt.  
$$ 

By the definition $\zt$ is the kernel of $ \pi$ and the homomorphism  $\pi $
maps the binomial $ \mathsf{b}$ into zero only if it maps 
\be \label{bin1}
u [\tilde \beta^+ ]- v [\tilde \beta^-]
\ee
into zero. Clearly, \eqref{bin1} can be zero under map \eqref{pi} only if 
$$
u w(\tilde \beta^+) +v w(\tilde \beta^-)=0.
$$
Therefore, without loss of generality we can assume that
\be \label{f_b}
f_\beta= w(\beta^-) [\beta^+]- w(\beta^+) [\beta^-],
\ee
where $supp(\beta^-) \cap supp(\beta^+)=\emptyset.   $

Let  $ \beta =\beta^+-\beta^-.  $
Since $supp(\beta^-) \cap supp(\beta^+)=\emptyset$, the notation is in the agreement with \eqref{beta_def}.
Denote by  $\mathsf{L}_\zeta$  the lattice 
$$
\mathsf{L}_\zeta= \{ \beta\in \Z^{n\ell} \ : \ \mathfrak{A}^{\n} \beta =0\}, 
$$
and, by $\mathsf{S}$ the polynomial ring 
$\Q(\zeta)[a]$, $\mathsf{S}=\Q(\zeta)[a]$. Furthermore,  let $\mathsf{S}_{\mathsf{a}}$ be the localization of  $\mathsf{S}$ 
with respect to the multiplicative set 
$\{ 1, \mathsf{a}, \mathsf{a}^2, \dots  \}$.

Since $f_\beta$ is of the form \eqref{f_b},  to prove that $f_\beta \in \zt$ it is sufficient to show that $\beta \in \mathsf{L}_\zeta.$
Since $f_\beta \in \zt : \mathsf{a}^\infty  $,
$$
\mathsf{a}^k (w(\beta^-) [\beta^+]  -w(\beta^+) [\beta^-]  )\in \zt
$$
for some $k\in \N_0$. Using Lemma \ref{lem424}, we obtain 
\be \label{bem1}
\frac{[\beta]}{w(\beta)}-1 \in \zt  \mathsf{S}_{\mathsf{a}}.
\ee
Moreover, it is easy to see that  
$$
\left \langle  \frac{[\gamma]}{w(\gamma)}-1 \ : \ \gamma \in \mathsf{L}_\zeta \right \rangle   =  \zt  \mathsf{S}_{\mathsf{a}}.
$$
Let 
$
F :  \mathsf{S}_{\mathsf{a}}  \to  \Q(\zeta)[\Z^{n\ell}/\mathsf{L}_\zeta]
$
be defined on the basis monomials of
$\mathsf{S}_{\mathsf{a}}$ by 
$$
F ([\gamma]/w(\gamma)) = \gamma+\mathsf{L}_\zeta.
$$
It is not difficult to see that  $F$ is a ring  epimorphism from  $\mathsf{S}_{\mathsf{a}}$ to the group ring  $\Q(\zeta)[G] $ where $G=\Z^{n\ell}/\mathsf{L}_\zeta$, and the kernel of $F$ is $  \zt  \mathsf{S}_{\mathsf{a}}. $ 
Then $ 
\mathsf{S}_{\mathsf{a}}/ (  \zt  \mathsf{S}_{\mathsf{a}} )  $  is isomorphic to $\Q(\zeta)[\Z^{n\ell}/\mathsf{L}_\zeta]$
and the isomorphism is defined 
by
$$
\frac{[\gamma]}{w(\gamma)} +  \zt  \mathsf{S}_{\mathsf{a}} \to \gamma+ \mathsf{L}_\zeta.
$$
By \eqref{bem1} under this isomorphism, the monomial on the left hand side of \eqref{bem1} is mapped to $1_G-(\beta +\mathsf{L}_\zeta)$ which is equal to zero of $\Q(\zeta)[G]$. Therefore $\beta \in \mathsf{L}_\zeta.$
\end{proof}

{
\begin{remark}
After the substitution 
$$
 a^{(j)}_{T^{j-1}_n(\pk)} \zeta^{-j}=  b^{(j)}_{T^{j-1}_n(\pk)}, 
$$
map \eqref{pi} becomes the usual toric map. Consequently, we predict that  most results on binomial (difference) ideals should remain true for $\zeta$-toric ideals as well.
\end{remark}
}

Now we prove the main result of our paper.

{\it Proof of Theorem    \ref{conj_2}}.
 According to Theorem \ref{thm427},   $ \zt : {\mathsf  a} ^\infty = \zt$. Furthermore, $I^{(n)}_R \subseteq \zt$. To see this, let $ \zeta^{|\mu|} [\mu]-[\hat \mu] \in I^{(n)}_R$, where $\mu \in \mathcal{M}\n$. 
By Lemma \ref{lem424} $ w(\hat \mu) [\mu]- w(\mu) [\hat \mu] \in I^{(n)}_R$.
Moreover,
  $$ w(\hat \mu) [\mu]- w(\mu)  [\hat \mu]= w(\theta) [\theta](w(\gamma) [\tau]- w(\tau)[\gamma])$$  where $\theta\in \N_0^{n\ell}$ and  $supp(\tau) \cap supp(\gamma) = \emptyset$. Let $\w = \mu - \hat \mu$. Then $\w^+ = \tau$ and $\w^- = \gamma$.  
    Since $\mathfrak{M}\n\w=0$ and 
    $$
    \sum_{s=1}^n \w_{n(i-1)+s}=\sum_{s=1}^n \mu_{n(i-1)+s}- \sum_{s=1}^n \hat \mu_{n(i-1)+s}=0 \quad {\rm for } \quad i=1,\dots, \ell,
    $$
    we have 
    $\mathfrak{A}\n  \w =0$. Hence by Proposition \ref{pro426}, $w(\hat \mu) [\mu]-w(\mu)  [\hat \mu] \in \zt$. 
    Thus,  $I^{(n)}_R \subseteq \zt$, and, therefore, $$ I^{(n)}_R: {\mathsf  a} ^\infty \subseteq \zt: {\mathsf  a} ^\infty = \zt.$$

    Next we show that  $\zt \subseteq I^{(n)}_R: {\mathsf  a} ^\infty$.  Let $w(\gamma)  [\tau]-w(\tau) [\gamma] \in \zt$. Since $ \zt: {\mathsf  a} ^\infty = \zt$, without loss of generality, we assume that $supp(\tau) \cap supp(\gamma) =\emptyset$. 
    
    {
    Now let $\w \in \N_0^{n\ell}$ be such that $\hat \tau + \hat \w = \gamma +\w$. To find such $\w$, notice that the latter equation is equivalent to $ \omega- \hat \omega= \hat \tau -\gamma, $ so 
    $$
    \omega= (\hat \tau -\gamma) (id -\widetilde {\bf T})^{-1},
    $$
where $id $ is the $n\ell \times n\ell$ identity matrix. 
 

    Let $\mu =\tau+\w$. Then $\hat \mu = \gamma+\w$ and $\mu-\hat\mu=\tau-\gamma$  yielding  $ \mathfrak{A}^{\n}\mu= \mathfrak{A}^{\n}\hat\mu $. Next we prove that  $\mu \in {\mathcal M}\n$. To this end, notice that from the definition of matrix } $\mathfrak{A}\n$ it follows that, 
    for $1 \leq i \leq n-1$,
    $$(L^i-L^{i+1})(\mu)= (L^i-L^{i+1})(\hat \mu),$$
    where $(L^i-L^{i+1})(\mu)$ is $L^i(\mu)-L^{i+1}(\mu)$. At the same time, it is easy to see that $$(L^{i-1}-L^{i})(\mu)= (L^i-L^{i+1})(\hat \mu),$$ for all $2 \leq i \leq n-1$.  Consequently, 
    for all $2 \leq i \leq n-1$,
    \begin{equation} 
        (L^i-L^{i+1})(\mu)= (L^{i-1}-L^{i})(\mu).
     \end{equation}  
    Furthermore, observe that
     $$
     \sum^{n-1}_{i=1}(L^i-L^{i+1})(\mu)=(L^1-L^n)(\mu)=-(L^1-L^2)(\mu)
     $$
    and, therefore, $n(L^i-L^{i+1})(\mu)=0$ or, equivalently, $(L^i-L^{i+1})(\mu)=0$  for all $2 \leq i \leq n-1$.
   Hence,   by Proposition \ref{pro_M}, $\mu \in {\mathcal M}\n$.
    Thus, {
    $$w(\hat \mu) [\mu] -w(\mu) [\hat \mu] = w(\w) [\w]\left(w(\gamma) [\tau]-w(\tau) [\gamma]\right) \in I^{(n)}_R.$$ This implies that $w(\gamma)[\tau]-w(\tau)[\gamma] \in I^{(n)}_R: {\mathsf  a} ^\infty$.  
    Therefore, $\zt  \subseteq I^{(n)}_R: {\mathsf  a} 
    ^\infty$. } $\hfill \square$

\subsection{Example: a cubic system}
Consider as an example cubic system \eqref{sys_3dim_ex}.
In this case, matrix \eqref{Mn} is 
$$
\widehat{ \mathfrak{M}}=
\begin{pmatrix} 
1 & -1 & 0 & 0 & 1 & -1 & 1 & 0 & -1 \\ 
0 & 1 & -1 & -1 & 0 & 1 & -1 & 1 & 0
\end{pmatrix}.
$$
The corresponding  toric ideal $ I_{\widehat{\mathfrak{M}}}$ is 
\begin{equation*}
\begin{array}{c} 
I_{\widehat{\mathfrak{M}}}= 
\la
-1 + a_{101} b_{110} c_{011}, c_{010} - b_{110} c_{011}, c_{001} - a_{101} c_{011}, 
 b_{100} - a_{101} b_{110},\\ b_{010} - b_{110} c_{011}, a_{100} - a_{101} b_{110}, 
 a_{001} - a_{101} c_{011} \ra.
\end{array}
\end{equation*}
To compute the Hilbert basis of the monoid corresponding to $\widehat {\mathfrak{M}}$, we use Algorithm
\ref{alg1} and consider the ideal 
\begin{equation*}
\begin{array}{c}
\la a_{100} - t_1 y_1, b_{010} - (t_2 y_2)/t_1, c_{001} -  y_3/t_2, a_{001} - y_4/t_2, b_{100} -  t_1 y_5, \\ c_{010} - (t_2 y_6)/t_1, a_{101} - (t_1 y_7)/t_2, b_{110} -  t_2 y_8, c_{011} - y_9/t_1  \ra.
\end{array}
\end{equation*}
Polynomials from the \gb of the above ideal   which do not depend on   
$t_1$ and $t_2$
are
\begin{equation*}
\begin{array}{c}
\{ a_{101} b_{110} c_{011} - y_7 y_8 y_9, -b_{110} c_{011} y_6 + c_{010} y_8 y_9, 
a_{101} c_{010} - y_6 y_7, -a_{101} b_{110} y_5 + b_{100} y_7 y_8,\\ 
b_{100} c_{011} - y_5 y_9, -b_{110} y_5 y_6 + b_{100} c_{010} y_8, 
a_{001} b_{110} - y_4 y_8, -a_{101} c_{011} y_4 + a_{001} y_7 y_9,\\ -c_{011} y_4 y_6 + 
 a_{001} c_{010} y_9, -a_{101} y_4 y_5 + a_{001} b_{100} y_7, a_{001} b_{100} c_{010} - y_4 y_5 y_6, 
b_{110} c_{001} - y_3 y_8, \\ -a_{101} c_{011} y_3 + c_{001} y_7 y_9, -a_{001} y_3 + 
 c_{001} y_4, -c_{011} y_3 y_6 + c_{001} c_{010} y_9, -a_{101} y_3 y_5 + b_{100} c_{001} y_7, \\
b_{100} c_{001} c_{010} - y_3 y_5 y_6, -b_{110} c_{011} y_2 + b_{010} y_8 y_9, -c_{010} y_2 + 
 b_{010} y_6, a_{101} b_{010} - y_2 y_7, \\-b_{110} y_2 y_5 + b_{010} b_{100} y_8, -c_{011} y_2 y_4 +
  a_{001} b_{010} y_9, a_{001} b_{010} b_{100} - y_2 y_4 y_5, \\-c_{011} y_2 y_3 + b_{010} c_{001} y_9,
 b_{010} b_{100} c_{001} - y_2 y_3 y_5, -a_{101} b_{110} y_1 + a_{100} y_7 y_8, -b_{100} y_1 + 
 a_{100} y_5, \\ a_{100} c_{011} - y_1 y_9, -b_{110} y_1 y_6 + a_{100} c_{010} y_8, -a_{101} y_1 y_4 +
  a_{001} a_{100} y_7, a_{001} a_{100} c_{010} - y_1 y_4 y_6,\\ -a_{101} y_1 y_3 + a_{100} c_{001} y_7,
 a_{100} c_{001} c_{010} - y_1 y_3 y_6, -b_{110} y_1 y_2 + a_{100} b_{010} y_8, \\
a_{001} a_{100} b_{010} - y_1 y_2 y_4, a_ {100} b_{010} c_{001} - y_1 y_2 y_3\}.
\end{array}
\end{equation*}
From these expressions, we see that the 
Hilbert basis of the   monoid $\widehat{\mcm}$ is
\begin{equation*}
\begin{array}{c}
\widehat H=\{(0,0,0,0,0,0,1,1,1), (0,0,0,0,0,1,1,0,0), 
(0,0,0,0,1,0,0,0,1),(0,0,0,1,0,0,0,1,0),\\
(0,0,0,1,1,1,0,0,0),(0,0,1,0,0,0,0,1,0),
(0,0,1,0,1,1,0,0,0),(0,1,0,0,0,0,1,0,0),\\
(0,1,0,1,1,0,0,0,0),(0,1,1,0,1,0,0,0,0),
(1,0,0,0,0,0,0,0,1),(1,0,0,1,0,1,0,0,0),\\
(1,0,1,0,0,1,0,0,0),(1,1,0,1,0,0,0,0,0),
(1,1,1,0,0,0,0,0,0)\}
\end{array}
\end{equation*}
and the Sibirsky ideal $\widehat I_S$ of system 
\eqref{sys_3dim_ex}  is 
\begin{equation*}
\begin{array}{c}
\widehat I_S=\la a_{101} c_{010} - a_{001} b_{110}, b_{100} c_{011} - c_{010} a_{101}, 
a_{001} b_{110} - b_{100} c_{011}, b_{110} c_{001} - a_{100} c_{011}, \\
b_{100} c_{001} c_{010} - a_{100} a_{001} c_{010}, a_{101} b_{010} - c_{001} b_{110}, 
a_{001} b_{010} b_{100} - c_{001} b_{100} c_{010},\\ b_{010} b_{100} c_{001} - a_{100} c_{001} c_{010}, 
a_{100} c_{011} - b_{010} a_{101}, a_{001} a_{100} c_{010} - b_{010} a_{001} b_{100}, \\
a_{100} c_{001} c_{010} - a_{100} b_{010} a_{001}, a_{001} a_{100} b_{010} - b_{010} c_{001} b_{100},\ra .
\end{array}
\end{equation*}

By \eqref{Izeta}, the  ideal $\widehat{\mathcal{I}}_\zeta$  is
computed as the fourth elimination ideal of 
\begin{equation*}
\begin{array}{c}
\la 1 - w \alpha \beta \gamma, 
-b_{010} + a_{100} \alpha \zeta, c_{010} - \alpha b_{100} \zeta, -c_{001} +  b_{010} \beta  \zeta, -c_{011} + \alpha b_{110} \beta \zeta, b_{100} -\\ a_{001} \gamma \zeta,
-b_{110} + 
 a_{101} \alpha \gamma \zeta, a_{001} - \beta  c_{010} \zeta, a_{100} - c_{001} \gamma \zeta, a_{101} - 
 \beta c_{011} \gamma \zeta, 1 + \zeta + \zeta^2\ra .
 \end{array}
\end{equation*}
Computations with {\sc Singular} \cite{sing}
over the field  $\Q(\zeta)$ yield
$$
\begin{array}{c}
\widehat I_\zeta=
\la  c_{001} b_{110}-\zeta  a_{100} c_{011},
   a_{001} b_{110}-\zeta  b_{100} c_{011}, 
   c_{010} a_{101}+(\zeta+1) b_{100} c_{011}, \\
   b_{010} a_{101}+(\zeta+1) a_{100} c_{011},
   b_{010} b_{100}-a_{100} c_{010}
   , a_{001} b_{010}-c_{001} c_{010}
    a_{001} a_{100}-b_{100} c_{001} 
\ra.
\end{array}
$$

Similarly,  
by \eqref{I_E}, the  ideal $\widehat{\mathcal{I}}_E$  is
computed as the fourth elimination ideal of 
\begin{equation*}
\begin{array}{c}
\la 1 - w \alpha \beta \gamma, a_{100} \alpha - b_{010}, a_{001} \gamma - b_{100}, a_{101} \alpha \gamma - b_{110}, 
b_{010} \beta - c_{001}, b_{100} \alpha - c_{010},\\
b_{110} \alpha \beta - c_{011}, c_{001} \gamma - a_{100}, 
c_{010} \beta - a_{001}, c_{011} \beta \gamma - a_{101} \ra 
\end{array}
\end{equation*}
and is equal  to 
\begin{equation*}
\begin{array}{c}
\widehat{\mathcal{I}}_E= \la -a_{101} c_{010} + b_{100} c_{011}, a_{001} b_{110} - a_{101} c_{010}, a_{101} b_{010} - b_{110} c_{001}, 
a_{001} b_{010} - c_{001} c_{010},\\ -b_{110} c_{001} + a_{100} c_{011}, -b_{010} b_{100} + a_{100} c_{010}, a_{001} a_{100} - b_{100} c_{001} \ra .
\end{array}
\end{equation*}

Computations with \texttt{Singular} 
show that
$$
\widehat I_S: (a\cdot b\cdot c)^\infty = \widehat{\mathcal{I}}_E,
$$
as it should be according to  Theorem  \ref{con_satur}.


The ideal $I\n_R$ can be immediately written down once we know  the ideal $I\n_S,   $ so in our case   using $\widehat I_S$  we obtain 
the corresponding ideal $I\n_R$, which we denote 
 $\widehat I_R$.
Computing  with \texttt{Singular} 
we see that 
$$
\widehat I_R: (a\cdot b\cdot c)^\infty = \widehat{\mathcal{I}}_\zeta,
$$
in agreement with Theorem \ref{conj_2}.

In Subsection \ref{sec_INF}, we have mentioned a relation of invariants to the structure of the normal form module. 
For $\zeta = (- 1 )^{ 2/3}$, the Poincar\'e-Dulac  normal form  of system \eqref{sys_3dim_ex}, up to order 4, is 
$$
\begin{aligned}
\dot x= & x + x^2 y z \left(\frac{1}{3} a_{001} a_{100} c_{010}  +\frac{2}{3} (-1)^{1/3} a_{001} a_{100} c_{010} 
+\frac{1}{3} (-1)^{2/3} a_{001} a_{100} c_{010} - \right. \\ & \left.(-1)^{1/3} a_{101} c_{010} +\frac 13 a_{001} b_{100} c_{010} - \frac{1}{3} (-1)^{1/3} a_{001} b_{100} c_{010} -\frac{2}{3} (-1)^{2/3} a_{001} b_{100} c_{010}\right)+\dots ,
\\
\dot y= & (- 1 )^{ 2/3} y+ x y^2 z \left(\frac{1}{3} (-1)^{2/3} a_{001} b_{010} b_{100}  -\frac{2}{3} a_{001} b_{010} b_{100} -\frac{1}{3} (-1)^{1/3} a_{001} b_{010} b_{100} +\right.  \\ & \left. (-1)^{2/3} a_{001} b_{110} + \frac{1}{3} a_{001} b_{100} c_{010} +\frac{2}{3} (-1)^{1/3} a_{001} b_{100} c_{010} + \frac{1}{3} (-1)^{2/3} a_{001} b_{100} c_{010} \right) +\dots  , 
\\
\dot z=& -(- 1 )^{ 1/3} z + x y z^2 
\left(  \frac{1}{3} b_{100} c_{001} c_{010}    -\frac{2}{3} a_{001} b_{100} c_{010}  - \frac{1}{3} (-1)^{1/3} a_{001} b_{100} c_{010} +\right.  \\ &
\left. \frac{1}{3} (-1)^{2/3} a_{001} b_{100} c_{010} -\frac{1}{3} (-1)^{1/3} b_{100} c_{001} c_{010} -\frac{2}{3} (-1)^{2/3} b_{100} c_{001} c_{010} +b_{100} c_{011} \right)  +\dots. 
\end{aligned}
$$
In particular, we see that 
the coefficients  of the normal form 
belong to the subalgebra generated by 
$$
\{ [\nu] : \nu \in \widehat H \},
$$
which is in agreement with Corollary \ref{cor_inv} and  \eqref{Stan_dec}.


{
From the results of this section and the example we see 
that the invariants are computed using the Lawrence lifting of the matrix $\mathfrak{M}^{\n}$, so the calculations are performing in an extended ring. 
To compute the equivariancy  conditions 
we use another lifting of the matrix 
$\mathfrak{M}^{\n}$, the matrix $\mathfrak{A}^{\n}$. The conditions of 
$\zeta$-reversibility are computed 
using the same lifting $\mathfrak{A}^{\n}$ and a kind of the "weight function" $w(\theta)$. 
Note that the matrix $\mathfrak{M}^{\n}$
is defined  using solely nonlinear 
terms of system \eqref{nSYS} and the matrices $O^{(k)}$ in 
$\mathfrak{A}^{\n}$ are defined using solely the linear part of \eqref{nSYS}.
It is an interesting and challenging 
problem to extend these results to the case of systems \eqref{sys_X} with more general matrix of the linear approximation. }

In conclusion, our study delves into the complex interplay between group actions, time-reversibility,  integrability of polynomial systems of ODEs, and binomial ideal theory. Our finding not only shed lights on these issues but also lay the groundwork for future exploration in this important  and promising direction  of research.

It would be particularly interesting to extend our analysis to systems described by \eqref{Asn}, especially those where the algebra of polynomial first integrals of \eqref{As} is generated by more than one element. A specific focus could be on examining the relationships between invariants and the normal form modules, along with their Stanley decomposition, for such systems.

\section*{Acknowledgments}   

Mateja Gra\v si\v c and Valery Romanovski acknowledge the support of the Slovenian Research and Innovation  Agency (core research programs P1-0288 and P1-0306, respectively).

\end{document}